\documentclass[11pt]{amsart}

\usepackage[pdftex,pdfpagelabels,bookmarks,hyperindex,hyperfigures]{hyperref}
  \hypersetup{
    colorlinks,
    linkcolor={red!55!black},
    urlcolor={blue!55!black},
    citecolor={green!55!black}
  }

\usepackage{amsmath}
\usepackage{amsthm}
\usepackage{amsfonts}
\usepackage{amssymb}
\usepackage{amsthm}
\usepackage{bbm}
\usepackage{enumitem}
\usepackage[mathscr]{euscript}
\usepackage{float}
\usepackage{ifthen}
\usepackage{listings}
\usepackage{mathdots}
\usepackage{mleftright}
\usepackage[new]{old-arrows}
\usepackage{rotating}
\usepackage{sansmath}
\usepackage{tabularx}
\usepackage{tikz}
  \usetikzlibrary{positioning}
\usepackage{tikz-cd}
  \tikzcdset{
    cells={font=\everymath\expandafter{\the\everymath\displaystyle}},
  }
\usepackage{todonotes}
\usepackage{xcolor}
\usepackage{xspace}

\usepackage[capitalize]{cleveref}

\definecolor{codegreen}{rgb}{0,0.6,0}
\definecolor{codegray}{rgb}{0.5,0.5,0.5}
\definecolor{codepurple}{rgb}{0.58,0,0.82}
\definecolor{backcolour}{rgb}{0.95,0.95,0.92}
\lstdefinestyle{mystyle}{
    xleftmargin={2.5em},
    backgroundcolor=\color{backcolour},
    commentstyle=\color{codegreen},
    keywordstyle=\color{magenta},
    numberstyle=\tiny\color{codegray},
    stringstyle=\color{codepurple},
    basicstyle=\ttfamily\footnotesize,
    breakatwhitespace=false,
    breaklines=true,
    captionpos=b,
    keepspaces=true,
    numbers=left,
    numbersep=5pt,
    showspaces=false,
    showstringspaces=false,
    showtabs=false,
    tabsize=2
}
\usepackage{subfiles}

\newtheorem{thm}{Theorem}
\numberwithin{thm}{section}

\newtheorem{prop}[thm]{Proposition}
\newtheorem{cnj}[thm]{Conjecture}

\theoremstyle{definition}

\newtheorem{cns}[thm]{Construction}

\newtheorem{exm}[thm]{Example}
\newtheorem{defn}[thm]{Definition}
\newtheorem{notn}[thm]{Notation}
\newtheorem{rmk}[thm]{Remark}
\newtheorem{warn}[thm]{Warning}

\newcommand{\bb}{\mathbb}

\newcommand{\s}{\mathscr}

\newcommand{\mc}{\mathcal}

\newcommand{\mcE}{\mc{E}}

\newcommand{\mcM}{\mc{M}}

\newcommand{\mcO}{\mc{O}}

\newcommand{\mcV}{\mc{V}}
\newcommand{\mcW}{\mc{W}}
\newcommand{\mcX}{\mc{X}}
\newcommand{\mcY}{\mc{Y}}
\newcommand{\mcZ}{\mc{Z}}

\newcommand{\sB}{\s{B}}

\newcommand{\sF}{\s{F}}
\newcommand{\sG}{\s{G}}

\newcommand{\sM}{\s{M}}

\newcommand{\sV}{\s{V}}

\newcommand{\bA}{\bb{A}}

\newcommand{\bC}{\bb{C}}

\newcommand{\bP}{\bb{P}}

\newcommand{\bV}{\bb{V}}

\newcommand{\bZ}{\bb{Z}}

\newcommand{\wh}{\widehat}
\newcommand{\br}[1]{\mleft( #1 \mright)}

\newcommand{\set}[2][]{
  \ifthenelse{\equal{#1}{}}{
    \mleft\{ #2 \mright\}
  }{
    \mleft\{ #1\ :\ #2 \mright\}
  }
}

\renewcommand{\to}[1][]{
  \ifthenelse{\equal{#1}{}}{
    \longrightarrow
  }{
    \stackrel{#1}{\longrightarrow}
  }
}

\newcommand{\To}[1][]{
  \ifthenelse{\equal{#1}{}}{
    \Longrightarrow
  }{
    \stackrel{#1}{\Longrightarrow}
  }
}

\renewcommand{\mapsto}[1][]{
    \ifthenelse{\equal{#1}{}}{
      \longmapsto
    }{
      \stackrel{#1}{\longmapsto}
    }
}

\newcommand{\ot}[1][]{
  \ifthenelse{\equal{#1}{}}{
  \longleftarrow
  }{
    \stackrel{#1}{\longleftarrow}
  }
}

\newcommand{\ol}{\overline}
\newcommand{\ul}{\underline}

\newcommand{\im}{\mathrm{im}\xspace}

\newcommand{\Gl}{\mathrm{Gl}}

\newcommand{\hto}[1][]{\stackrel{#1}{\longhookrightarrow}}

\newcommand{\id}{\mathrm{id}}

\renewcommand{\Ref}[2][]{\ifthenelse{\equal{#1}{}}{\ref{#2}}
                      {\hyperref[#2]{\ref*{#1}(\ref*{#2})}}}
\newcommand{\Aref}[2][]{\ifthenelse{\equal{#1}{}}{\autoref{#2}}
                      {\hyperref[#2]{\autoref*{#1}\ref*{#2}}}}
\newcommand{\Sref}[1]{\hyperref[#1]{\S \ref*{#1}}}
\newcommand{\dom}{\mathrm{dom}}

\newcommand{\Set}{\mathrm{Set}}

\newcommand{\Grpd}{\mathrm{Grpd}}
\newcommand{\op}{\mathrm{op}}

\newcommand{\PSh}{\mathrm{PSh}}

\newcommand{\Sh}{\mathrm{Sh}}
\newcommand{\Fun}{\mathrm{Fun}}

\newcommand{\Hom}{\mathrm{Hom}}
\newcommand{\Nat}{\mathrm{Nat}}

\newcommand{\HHom}{\mathcal{H}\mathrm{om}}
\newcommand{\Map}{\mathrm{Map}}

\newcommand{\Sc}[1][]{\ifthenelse{\equal{#1}{}}{\mathrm{Sch}}{\mathrm{Sch}_{/#1}}}

\newcommand{\Spec}{\mathrm{Spec}}

\newcommand{\Spf}{\mathrm{Spf}}

\newcommand{\Cat}{\s{C}\mathrm{at}}

\newcommand{\QCoh}{\mathrm{QCoh}}
\newcommand{\Vect}{\mathrm{Vect}}

\newcommand{\Perf}{\mathrm{Perf}}

\newcommand{\Sym}{\mathrm{Sym}\xspace}
\newcommand{\St}{\mathrm{St}}
\newcommand{\PSt}{\mathrm{PSt}}
\newcommand{\AlgSt}{\mathrm{AlgSt}}
\newcommand{\Sch}{\mathrm{Sch}}
\newcommand{\FSch}{\mathbb{F}\mathrm{Sch}}

\newcommand{\Mod}{\mathrm{Mod}}
\newcommand{\LMod}{\mathrm{LMod}}
\newcommand{\RMod}{\mathrm{RMod}}

\newcommand{\CAlg}{\mathrm{cAlg}}
\newcommand{\CCAlg}{\mathrm{c}\mathcal{A}\mathrm{lg}}

\newcommand{\Ob}[1]{\mathrm{Ob}\br{#1}}

\newcommand{\Aff}{\mathrm{Aff}}

\newcommand{\stb}{\mathrm{stb}}

\setenumerate{itemsep=1pt, label=(\roman*)}
\setitemize{itemsep=1pt, leftmargin=1.5em}

\usepackage[style=alphabetic]{biblatex}

\AtNextBibliography{\small}

\usepackage[margin=1.05in]{geometry}

\title[Moduli stacks of quiver connections]
      {Moduli stacks of quiver connections and non-Abelian Hodge theory}

\author{Mahmud Azam}
\email{mahmud.azam@usask.ca}

\author{Steven Rayan}
\email{rayan@math.usask.ca}

\address{Centre for Quantum Topology and Its Applications (quanTA) and
Department of Mathematics and Statistics, University of Saskatchewan, SK,
Canada~ S7N 5E6}

\begin{document}

\begin{abstract}
In \cite{ModQuivBun}, a moduli stack parametrizing  $I$--indexed diagrams of Higgs bundles over a base stack $X$ was constructed for any finite simplicial set $I$, inspiring speculations about extending the non-Abelian Hodge correspondence to these moduli stacks. In the present work, we formalize the de Rham side of this conjectural extension. We construct moduli stacks parametrizing diagrams of bundles with $\lambda$--connections over a base prestack $X$, where $\lambda$ can be a fixed number or a parameter. Taking $\lambda$ to be $1$ gives a moduli
stack parametrizing diagrams of bundles with connection, while taking it to be a
parameter gives a version of Simpson's non-Abelian Hodge filtration for digrams
of bundles with connection. We show that when $X$ is a smooth and projective
scheme over an algebraically closed field $k$ of characteristic $0$, these
moduli stacks are algebraic and locally of finite presentation, and have affine
diagonal.
\end{abstract}

\maketitle

\tableofcontents


\section{Introduction}

\subsection{Motivation}

Representations of the fundamental group of various algebraic varieties make
prominent appearances in various areas of mathematics and physics. A few
examples are in order:

\begin{itemize}
\item We recall that the fundamental group of the configuration space of $n$
points on the plane is the Artin braid group. Representations of this
group characterize solutions to the quantum Yang-Baxter equations
\cite{BirmanKnotThry} and they govern the statistics of anyon exchanges
\cite{GoldinAnyons}. This makes them important in topological quantum computing
\cite{KauffmanBraidRepQGates} in the sense of \cite{TopPhasesQC}.

\item Some recent advances in condensed matter physics, namely hyperbolic band
theory \cites{HyperbolicBandTheory, AutBlochTheorems, HypBloch}, are centred around
representations of some choice of discrete translation group acting on the
hyperbolic plane. In turn, these translation groups are the fundamental groups
of the Riemann surfaces obtained by quotienting the hyperbolic plane by the
translation action.

\item A large class of $d = 4, n = 2$ field theories called ``class S''
\cite{ClassS} obtained by a specific compactification method applied to
$d = 6, N = (2, 0)$ superconformal field theories, are controlled by
choices of representations of the fundmanetal group of the surface
parametrizing the ``compactified'' dimensions.
\end{itemize}

A series of equivalences, then, allows us to address representations of
fundamental groups in terms of vector bundles and connections as we briefly
discuss now. There is a well-known equivalence of categories between the
category of representations of the fundamental group of a ``nice enough''
topological space, such as a smooth manifold, and the category of locally
constant sheaves over the space. This equivalence restricts to an equivalence
between linear representations of the fundamental group on one side and locally
constant sheaves of vector spaces on the other. Then, the Riemann-Hilbert
correspondence (see, for example, \cite[35]{LanglandsIntro}) gives an
equivalence of categories between the category of locally constant sheaves and
the category of vector bundles with flat connections. On the other hand, there
is an equivalence between the category of vector bundles with flat connections
on a base complex manifold and the category of what are called Higgs bundles on
the base. This was first shown by Hitchin in \cite{SelfDualityEqn} as an
isomorphism of the moduli spaces of the respective objects when the base is a
curve, and was a generalization of the Narasimhan-Seshadri theorem. It was then
generalized to arbitrary compact K\"ahler manifolds by Simpson in
\cite{HiggsLocSys}. This equivalence is called the Corelette-Simpson
correspondence or the non-Abelian Hodge correspondence. For convenience, we will
call the study of these equivalences as non-Abelian Hodge theory, even though
the term is generally used only in association with the Corlette-Simpson
correspondence.

We can thus argue that
understanding categories and moduli spaces of flat connections and those of
Higgs bundles are of importance simultaneously to geometry, physics and
quantum computing. This work is a contribution in this general endeavour in
that it is a first step in unifying the moduli theoretic and the
categorical perspectives of the above discussed correspondences in a very
specific manner.
The moduli spaces allow us to vary the objects in some geometric way, but one
should be able to do the same with the morphisms, and the categorical
equivalences should be ``geometric'' in some good sense. Once we try to make
this idea precise, we arrive at the following situation:
we should have moduli spaces of objects and moduli spaces of morphisms
forming category objects internal to some geometric category, and the
equivalences should be equivalences of internal categories in the geometric
category, instead of isomorphisms of moduli spaces and equivalences of
categories running in parallel, apparently unaware of each other.
We can begin to visualize the situation as follows:
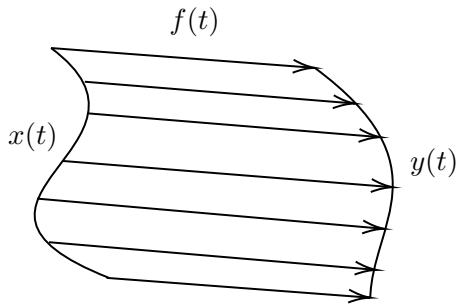
\begin{figure}[H]
\centering
\tikzset{every picture/.style={line width=0.75pt}} 

\begin{tikzpicture}[x=0.75pt,y=0.75pt,yscale=-1,xscale=1]

\draw    (100,123) .. controls (166.25,172) and (30.25,200) .. (129.25,239) ;
\draw    (233,133) .. controls (299.25,182) and (262,212) .. (261,249) ;
\draw    (100,123) -- (231.01,132.85) ;
\draw [shift={(233,133)}, rotate = 184.3] [color={rgb, 255:red, 0; green, 0; blue, 0 }  ][line width=0.75]    (10.93,-3.29) .. controls (6.95,-1.4) and (3.31,-0.3) .. (0,0) .. controls (3.31,0.3) and (6.95,1.4) .. (10.93,3.29)   ;
\draw    (106,180) -- (270.42,193.09) ;
\draw [shift={(272.42,193.25)}, rotate = 184.55] [color={rgb, 255:red, 0; green, 0; blue, 0 }  ][line width=0.75]    (10.93,-3.29) .. controls (6.95,-1.4) and (3.31,-0.3) .. (0,0) .. controls (3.31,0.3) and (6.95,1.4) .. (10.93,3.29)   ;
\draw    (119,156) -- (264.26,167.84) ;
\draw [shift={(266.25,168)}, rotate = 184.66] [color={rgb, 255:red, 0; green, 0; blue, 0 }  ][line width=0.75]    (10.93,-3.29) .. controls (6.95,-1.4) and (3.31,-0.3) .. (0,0) .. controls (3.31,0.3) and (6.95,1.4) .. (10.93,3.29)   ;
\draw    (117,140) -- (251.26,150.84) ;
\draw [shift={(253.25,151)}, rotate = 184.62] [color={rgb, 255:red, 0; green, 0; blue, 0 }  ][line width=0.75]    (10.93,-3.29) .. controls (6.95,-1.4) and (3.31,-0.3) .. (0,0) .. controls (3.31,0.3) and (6.95,1.4) .. (10.93,3.29)   ;
\draw    (93,199) -- (266.09,214.08) ;
\draw [shift={(268.08,214.25)}, rotate = 184.98] [color={rgb, 255:red, 0; green, 0; blue, 0 }  ][line width=0.75]    (10.93,-3.29) .. controls (6.95,-1.4) and (3.31,-0.3) .. (0,0) .. controls (3.31,0.3) and (6.95,1.4) .. (10.93,3.29)   ;
\draw    (99,221) -- (262.26,234.83) ;
\draw [shift={(264.25,235)}, rotate = 184.84] [color={rgb, 255:red, 0; green, 0; blue, 0 }  ][line width=0.75]    (10.93,-3.29) .. controls (6.95,-1.4) and (3.31,-0.3) .. (0,0) .. controls (3.31,0.3) and (6.95,1.4) .. (10.93,3.29)   ;
\draw    (128,239) -- (259.01,248.85) ;
\draw [shift={(261,249)}, rotate = 184.3] [color={rgb, 255:red, 0; green, 0; blue, 0 }  ][line width=0.75]    (10.93,-3.29) .. controls (6.95,-1.4) and (3.31,-0.3) .. (0,0) .. controls (3.31,0.3) and (6.95,1.4) .. (10.93,3.29)   ;

\draw (76.92,159.07) node [anchor=north west][inner sep=0.75pt]    {$x( t)$};
\draw (279.58,172.07) node [anchor=north west][inner sep=0.75pt]    {$y( t)$};
\draw (158.25,101.07) node [anchor=north west][inner sep=0.75pt]    {$f( t)$};
\end{tikzpicture}
\end{figure}
\noindent Here, $x(t)$ and $y(t)$ are two paths in the moduli space of objects
and $f(t)$ is a path in the moduli space of morphisms, such that
for each $t$, $f(t)$ has source $x(t)$ and target $y(t)$. At this point,
it is imaginable to have not just moduli spaces of morphisms but those
parametrizing diagrams of any shape. We may draw the situation for a diagram
of the shape $\Delta^2$ --- the commuting triangle --- as:
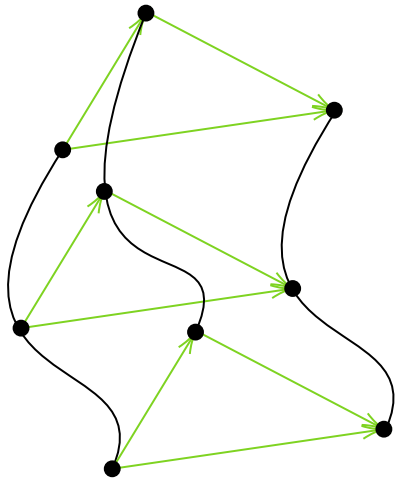
\begin{figure}[H]
\centering

\tikzset{every picture/.style={line width=0.75pt}} 

\begin{tikzpicture}[x=0.75pt,y=0.75pt,yscale=-1,xscale=1]

\draw [color={rgb, 255:red, 126; green, 211; blue, 33 }  ,draw opacity=1 ]   (246,64) -- (338.23,112.32) ;
\draw [shift={(340,113.25)}, rotate = 207.65] [color={rgb, 255:red, 126; green, 211; blue, 33 }  ,draw opacity=1 ][line width=0.75]    (10.93,-3.29) .. controls (6.95,-1.4) and (3.31,-0.3) .. (0,0) .. controls (3.31,0.3) and (6.95,1.4) .. (10.93,3.29)   ;
\draw [color={rgb, 255:red, 126; green, 211; blue, 33 }  ,draw opacity=1 ]   (203,133.25) -- (338.02,113.54) ;
\draw [shift={(340,113.25)}, rotate = 171.69] [color={rgb, 255:red, 126; green, 211; blue, 33 }  ,draw opacity=1 ][line width=0.75]    (10.93,-3.29) .. controls (6.95,-1.4) and (3.31,-0.3) .. (0,0) .. controls (3.31,0.3) and (6.95,1.4) .. (10.93,3.29)   ;
\draw [color={rgb, 255:red, 126; green, 211; blue, 33 }  ,draw opacity=1 ]   (203,133.25) -- (243.96,65.96) ;
\draw [shift={(245,64.25)}, rotate = 121.33] [color={rgb, 255:red, 126; green, 211; blue, 33 }  ,draw opacity=1 ][line width=0.75]    (10.93,-3.29) .. controls (6.95,-1.4) and (3.31,-0.3) .. (0,0) .. controls (3.31,0.3) and (6.95,1.4) .. (10.93,3.29)   ;
\draw  [fill={rgb, 255:red, 0; green, 0; blue, 0 }  ,fill opacity=1 ] (336.25,113.25) .. controls (336.25,111.18) and (337.93,109.5) .. (340,109.5) .. controls (342.07,109.5) and (343.75,111.18) .. (343.75,113.25) .. controls (343.75,115.32) and (342.07,117) .. (340,117) .. controls (337.93,117) and (336.25,115.32) .. (336.25,113.25) -- cycle ;
\draw  [fill={rgb, 255:red, 0; green, 0; blue, 0 }  ,fill opacity=1 ] (241.25,64.25) .. controls (241.25,62.18) and (242.93,60.5) .. (245,60.5) .. controls (247.07,60.5) and (248.75,62.18) .. (248.75,64.25) .. controls (248.75,66.32) and (247.07,68) .. (245,68) .. controls (242.93,68) and (241.25,66.32) .. (241.25,64.25) -- cycle ;
\draw  [fill={rgb, 255:red, 0; green, 0; blue, 0 }  ,fill opacity=1 ] (199.25,133.25) .. controls (199.25,131.18) and (200.93,129.5) .. (203,129.5) .. controls (205.07,129.5) and (206.75,131.18) .. (206.75,133.25) .. controls (206.75,135.32) and (205.07,137) .. (203,137) .. controls (200.93,137) and (199.25,135.32) .. (199.25,133.25) -- cycle ;
\draw [color={rgb, 255:red, 126; green, 211; blue, 33 }  ,draw opacity=1 ]   (225,154) -- (317.23,202.32) ;
\draw [shift={(319,203.25)}, rotate = 207.65] [color={rgb, 255:red, 126; green, 211; blue, 33 }  ,draw opacity=1 ][line width=0.75]    (10.93,-3.29) .. controls (6.95,-1.4) and (3.31,-0.3) .. (0,0) .. controls (3.31,0.3) and (6.95,1.4) .. (10.93,3.29)   ;
\draw [color={rgb, 255:red, 126; green, 211; blue, 33 }  ,draw opacity=1 ]   (182,223.25) -- (317.02,203.54) ;
\draw [shift={(319,203.25)}, rotate = 171.69] [color={rgb, 255:red, 126; green, 211; blue, 33 }  ,draw opacity=1 ][line width=0.75]    (10.93,-3.29) .. controls (6.95,-1.4) and (3.31,-0.3) .. (0,0) .. controls (3.31,0.3) and (6.95,1.4) .. (10.93,3.29)   ;
\draw [color={rgb, 255:red, 126; green, 211; blue, 33 }  ,draw opacity=1 ]   (182,223.25) -- (222.96,155.96) ;
\draw [shift={(224,154.25)}, rotate = 121.33] [color={rgb, 255:red, 126; green, 211; blue, 33 }  ,draw opacity=1 ][line width=0.75]    (10.93,-3.29) .. controls (6.95,-1.4) and (3.31,-0.3) .. (0,0) .. controls (3.31,0.3) and (6.95,1.4) .. (10.93,3.29)   ;
\draw  [fill={rgb, 255:red, 0; green, 0; blue, 0 }  ,fill opacity=1 ] (315.25,203.25) .. controls (315.25,201.18) and (316.93,199.5) .. (319,199.5) .. controls (321.07,199.5) and (322.75,201.18) .. (322.75,203.25) .. controls (322.75,205.32) and (321.07,207) .. (319,207) .. controls (316.93,207) and (315.25,205.32) .. (315.25,203.25) -- cycle ;
\draw  [fill={rgb, 255:red, 0; green, 0; blue, 0 }  ,fill opacity=1 ] (220.25,154.25) .. controls (220.25,152.18) and (221.93,150.5) .. (224,150.5) .. controls (226.07,150.5) and (227.75,152.18) .. (227.75,154.25) .. controls (227.75,156.32) and (226.07,158) .. (224,158) .. controls (221.93,158) and (220.25,156.32) .. (220.25,154.25) -- cycle ;
\draw  [fill={rgb, 255:red, 0; green, 0; blue, 0 }  ,fill opacity=1 ] (178.25,223.25) .. controls (178.25,221.18) and (179.93,219.5) .. (182,219.5) .. controls (184.07,219.5) and (185.75,221.18) .. (185.75,223.25) .. controls (185.75,225.32) and (184.07,227) .. (182,227) .. controls (179.93,227) and (178.25,225.32) .. (178.25,223.25) -- cycle ;
\draw    (228,294.25) .. controls (256,231.25) and (119.25,263) .. (203,133.25) ;
\draw [color={rgb, 255:red, 126; green, 211; blue, 33 }  ,draw opacity=1 ]   (271,225) -- (363.23,273.32) ;
\draw [shift={(365,274.25)}, rotate = 207.65] [color={rgb, 255:red, 126; green, 211; blue, 33 }  ,draw opacity=1 ][line width=0.75]    (10.93,-3.29) .. controls (6.95,-1.4) and (3.31,-0.3) .. (0,0) .. controls (3.31,0.3) and (6.95,1.4) .. (10.93,3.29)   ;
\draw [color={rgb, 255:red, 126; green, 211; blue, 33 }  ,draw opacity=1 ]   (228,294.25) -- (363.02,274.54) ;
\draw [shift={(365,274.25)}, rotate = 171.69] [color={rgb, 255:red, 126; green, 211; blue, 33 }  ,draw opacity=1 ][line width=0.75]    (10.93,-3.29) .. controls (6.95,-1.4) and (3.31,-0.3) .. (0,0) .. controls (3.31,0.3) and (6.95,1.4) .. (10.93,3.29)   ;
\draw [color={rgb, 255:red, 126; green, 211; blue, 33 }  ,draw opacity=1 ]   (228,294.25) -- (268.96,226.96) ;
\draw [shift={(270,225.25)}, rotate = 121.33] [color={rgb, 255:red, 126; green, 211; blue, 33 }  ,draw opacity=1 ][line width=0.75]    (10.93,-3.29) .. controls (6.95,-1.4) and (3.31,-0.3) .. (0,0) .. controls (3.31,0.3) and (6.95,1.4) .. (10.93,3.29)   ;
\draw  [fill={rgb, 255:red, 0; green, 0; blue, 0 }  ,fill opacity=1 ] (361.25,274.25) .. controls (361.25,272.18) and (362.93,270.5) .. (365,270.5) .. controls (367.07,270.5) and (368.75,272.18) .. (368.75,274.25) .. controls (368.75,276.32) and (367.07,278) .. (365,278) .. controls (362.93,278) and (361.25,276.32) .. (361.25,274.25) -- cycle ;
\draw  [fill={rgb, 255:red, 0; green, 0; blue, 0 }  ,fill opacity=1 ] (266.25,225.25) .. controls (266.25,223.18) and (267.93,221.5) .. (270,221.5) .. controls (272.07,221.5) and (273.75,223.18) .. (273.75,225.25) .. controls (273.75,227.32) and (272.07,229) .. (270,229) .. controls (267.93,229) and (266.25,227.32) .. (266.25,225.25) -- cycle ;
\draw  [fill={rgb, 255:red, 0; green, 0; blue, 0 }  ,fill opacity=1 ] (224.25,294.25) .. controls (224.25,292.18) and (225.93,290.5) .. (228,290.5) .. controls (230.07,290.5) and (231.75,292.18) .. (231.75,294.25) .. controls (231.75,296.32) and (230.07,298) .. (228,298) .. controls (225.93,298) and (224.25,296.32) .. (224.25,294.25) -- cycle ;
\draw    (366,274.25) .. controls (394,211.25) and (257.25,243) .. (341,113.25) ;
\draw    (270,225.25) .. controls (298,162.25) and (177.25,232) .. (245,64.25) ;

\end{tikzpicture}
\end{figure}

In this work, we make these ideas precise: we construct and establish
the geometricity of the respective moduli spaces of diagrams of vector
bundles with connection. Previously, we constructed and studied similar moduli
spaces of diagrams of vector bundles and Higgs bundles \cite{ModQuivBun}.
In combination with the current work, this gives a unification of the moduli
theory and the category theory of all sides involved in non-Abelian Hodge
theory.
The diagrams we are considering can be thought of as representations of quivers
in the sense of \cite{QuivRepKirillov} but in the categories of vector bundles,
Higgs bundles, connections or $\lambda$--connections over a fixed base,
as opposed to the category of vector spaces --- hence, the term ``quiver
bundle''.
At the same time, we provide ``face'' and ``degeneracy'' maps, in the same sense
as in the context of simplicial sets. That is, a face map sends a diagram to one
of its lower dimensional faces, while a degeneracy map sends a diagram to a
higher dimensional diagram obtained by adding identity edges. The collection
of these results should be thought of as the beginnings of a form of
categorification of non-Abelian Hodge theory.

\subsection{Overview and Main Results}

We set up the basic notation and conventions, and also prove some basic results
about prestacks and stacks in \cref{sec:Prelim}. These should be well-known but
are difficult to find explicit descriptions of in the literature.

In \cref{sec:RelSpec}, we will give a detailed description of the relative
spectrum construction for prestacks, and show that it is given by a certain
Grothendieck construction or unstraightening.  Again, it is difficult to find
this exact description in the literature but will be a technical requirement for
proving our main results in the next sections.

In \cref{sec:ArrowBun}, For a prestack $\mcY$ over $S$, we define a prestack
$\sM_1(\mcY)$ whose objects over an $S$--scheme $U$ are triples $(E, F, s)$
where $E, F$ are vector bundles, or equivalently, finite locally free sheaves on
$U \times_S \mcY$ and $s : E \to F$ is a morphism of vector bundles. This
already appeared in our previous work \cite{ModQuivBun} but in the present
paper, we give a much more concrete description of this prestack, which will be
necessary for proving our main results. We call this prestack the moduli
prestack of arrow bundles on $\mcY$. As shown in \cite{ModQuivBun}, it follows
from our definition that when $\mcY$ is a stack, so is $\sM_1(\mcY)$, and when
$\mcY$ is algebraic and satisfies some mild conditions, then $\sM_1(\mcY)$ is
also algebraic and satisfies similar conditions.

Our main results are in \cref{sec:ArrowBunFormalGrpd}.  Here, for a scheme $X$
over $S$, we consider various formal groupoid structures
$\sF \rightrightarrows X$ on $X$ and the moduli stack $\sM_1(X_\sF)$ of arrow
bundles on $X_\sF$ --- where $X_\sF$ is the quotient stack of the formal
groupoid $(X, \sF)$. We show that the objects of this moduli stack over an
$S$--scheme $U$ are triples $(E, \phi, F, \psi, s)$, where $E, F$ are vector
bundles on $U \times_S X$; $\phi, \psi$ are $\lambda$--connections on $E, F$
respectively --- with $\lambda$ being either $0, 1$ or a morphism
$\lambda : U \to \bA^1_k$, depending on the choice of $\sF$ --- and
$s : E \to F$ is a morphism of $\lambda$--connections, in the sense that it is a
morphism of vector bundles making a certain square involving $\phi$ and $\psi$
commute.  Strictly speaking, we will be dealing with an equivalent formulation
of connections: a module over a sheaf of rings of differential operators
\cite[\S 2]{ModRepFunGrpI}.  We show that as long as the moduli stack $\sM_1(X)$
of arrow bundles on $X$, the moduli stack $\sM(X)$ of vector bundles on $X$ and
the moduli stacks $\sM(X_\sF)$ of vector bundles on $X_\sF$ constructed in
\cite[31---33]{NonAbHodgeFilt} are algebraic, then so is $\sM_1(X_\sF)$ --- see
\cref{thm:mod-st-arr-bun-formal-grpd-alg}.  Furthermore, in the case that $X$ is
a smooth and projective scheme, $\sM(X), \sM_1(X)$ are known to be algebraic,
locally of finite presentation and possessing affine diagonal
\cites{Wang-BunG, ModQuivBun}, and we show that $\sM_1(X_\sF)$ inherits this
good behaviour --- see \cref{thm:mod1-Lambda-alg-aff-diag-loc-fp} and
\cref{thm:mod-st-conn-alg-lfp-aff-diag}. Of course, taking $\lambda$ to be
$0$, $1$ and a $k$--valued parameter respectively, this gives a certain
``categorification'' of the moduli stacks of Higgs bundles and connections, and
of the non-Abelian Hodge filtration from \cite[31---33]{NonAbHodgeFilt}.

At the same time, there are moduli stacks $\sM_I(X_\sF)$ parametrizing
$I$--indexed diagrams of $\lambda$--connections over a scheme $X$ --- these are
just the moduli stacks $\mcM_{\Vect(X_\sF), I}$ of $I$--shaped quiver bundles
$X_\sF$ constructed in \cite[\S 4]{ModQuivBun}.  Importantly, this construction
is contravariantly functorial in $I$, by definition, so that we get ``face'' and
``degeneracy'' maps $\sM_{\Delta^i}(X_\sF) \to \sM_{\Delta^j}(X_\sF)$
corresponding to simplicial maps $\Delta^j \to \Delta^i$ as would be expected
from a categorification.  We discuss this in \cref{subsec:QuivBun} and show
that, when $X$ is smooth and projective, for any finite simplicial set $I$, the
moduli stacks $\sM_I(X_\sF)$ are algebraic, locally of finite presentation and
have affine diagonal. This is \cref{thm:mod-quiv-bun-formal-grpd-loc-fin-pre}.
Again, by varying $\sF$, we get the corresponding moduli stacks for Higgs
bundles, connections, or the non-Abelian Hodge filtration. In
\cref{subsec:NonAbHodge}, we speculate about a version of the non-Abelian Hodge
correspondence in this categorified setting using the categorified non-Abelian
Hodge filtration.

\begin{rmk}\label{rmk:Porta-Sala}
We must note that derived enhancements of these moduli stacks can be obtained by
\cite[Proposition II.2.8(3)]{DPS25}. Then, \cite[Corollary 4.6]{PS23} combined
with \cite[Proposition 4.1.1 (3) and (4)]{PS20} yield the geometricity and local
finite presentation of the derived moduli stack on the de Rham side.
Similarly, \cite[Corollary 4.6]{PS23} and \cite[Proposition 5.3.1]{PS20} prove
the geometricity and local finite presentation of the derived moduli stack on
the Dolbeault side. Similar results should hold for the derived non-Abelian Hodge
filtration. While these results work in far greater generality and imply the algebraicity
of the (non-derived) moduli stacks we study in this paper, our constructions and methods
of proof yield some more geometric properties of
the relevant moduli stacks, such as affineness of the diagonal, under suitable hypotheses on
the base scheme.
\end{rmk}

\subsection*{Acknowledgements}

We would like to thank Gabriel Ribeiro and Qixiang Wang for several
formative conversations about the de Rham space of a scheme, as well as Dat Minh Ha
for many helpful conversations about relevant algebro-geometry constructions. We are
indebted to Carlos Simpson for engaging in discussion of some key ideas at the frontiers of non-abelian Hodge theory.
We are grateful to Mauro Porta and Francesco Sala for pointing us to their work on moduli stacks of
diagrams of sheaves on the de Rham, Dolbeault and Deligne shapes, and clarifying for us the connection of their work with ours. We are grateful to Toni Annala for suggesting the use of box tensor products of universal vector bundles in the present and previous work. We are especially indebted to Antoine Bourget for hosting the first-named author at CEA Saclay for a research stay in Autumn 2024 during which some of the starting ideas of this work were formed, as well as several useful conversations about quiver varieties. We are also grateful to Campus France and the Embassy of France in Canada for a High Level Scientific Fellowship that funded this stay. The second-named author is grateful to Tony Pantev for very useful discussions during a visit to the University of Pennsylvania in December 2025. Throughout this work, the first named-author was funded by a Canada Graduate Scholarship (Doctoral) from the Natural Sciences and Engineering Research Council of Canada (NSERC) and the second-named author was funded by an NSERC Discovery Grant.


\section{Preliminaries}\label{sec:Prelim}

\subsection{Category Theory Conventions}

For category theoretic notation, we try to stay as close to \cite{HTT} as
possible. For example, we write $\Delta^n$ for the poset of natural numbers
$\set{0 < \dots < n}$ which we consider as a category. We will also write
$\Delta^n$ to denote its nerve or the standard $n$--dimensional
simplicial set, when needed.
For two categories $C, D$, we write $\Fun(C, D)$ for the
category of functors $C \to D$, and $C^\simeq$ for the core or maximal
subgroupoid of $C$. For a category $C$ and two objects
$c, d \in C$, we will write $\Hom_C(c, d)$ or $C(c, d)$ to denote
the set of morphisms $c \to d$ in $C$.
We will write a category object internal to a category $C$
as a tuples $(O, M, s, t, c, i)$, where $O$ is the object of objects,
$M$ is the object of morphisms, $s, t, c, i$ are the source, target,
composition and identity maps respectively.

We will use the following notation for $2$--categorical notions.
Given a pair of objects $a, b$ and $1$--morphisms $f, g : a \to b$ in a
$2$--category, we will denote a $2$--morphism $\alpha$ from $f$ to $g$ as
$\alpha : f \To g : a \to b$. We will denote by $- \circ -$, the composition of
$1$--morphisms as usual but also vertical composition. We will denote
horizontal composition by $- \star -$. We will denote the identity $2$--morphism
of a $1$--morphism $f$ as $e_f$, but in the case that $f = \id_a$ for some
object, we will simply write $e_a$.

\subsection{Prestacks and Stacks}

Throughout this work, we will fix an algebraically closed field $k$.
Let $\Sch$ denote the $1$--category of schemes over $k$ and $\Aff$ denote the
full subcategory of affine schemes over $k$.
We will denote by $\PSt$ the $2$--category of
prestacks or categories fibred in groupoids over $\Aff$, and by
$\St$, the full sub-$2$-category thereof consisting of stacks with respect
to the \'etale topology.
We recall that \'etale stacks are also fppf stacks so that it
makes sense to speak of Artin \'etale stacks, which we will simply refer to as
algebraic stacks. We denote the $2$--category of Artin \'etale stacks
or algebraic stacks as $\AlgSt$ which is a full sub-$2$-category of the
$2$--category of all Artin fppf stacks over $k$.
For a scheme $S \in \Sch$, we will write $S$ for
both the scheme and its image in $\PSt$ or $\St$ under the fully faithful Yoneda
embedding $\Sch \to \St \to \PSt$ --- recall that the embedding factors through
$\St$ since the \'etale topology is subcanonical.
For most of this work, we will fix a scheme $S$ and work
in the slice $2$--categories $\PSt_{/S}$ and $\St_{/S}$.

Since the Yoneda embedding for stacks and the embedding of stacks in prestacks
preserves limits, we will simply write the word ``limit'' for both
the $1$--limit of schemes and the $2$--limit of (pre)stacks, and similarly for
fibre products. By ``colimit'' of prestacks, we will generally mean the
$2$--colimit. We denote the fibre product functor over a prestack $\mcW$ as
$- \times_\mcW -$, but $- \times -$, suppressing $\mcW$, when $\mcW = S$.
Given prestacks $\mcY_1, \dots, \mcY_n \in \PSt_{/S}$, we write
$pr_i : \mcY_1 \times \cdots \times \mcY_n \to \mcY_i$ to denote the
projection of the iterated fibre product onto the $i$--th factor.
We will write $\mcY^n$ for the $n$--fold fibre product of $\mcY$ with itself
over $S$ $n$ times.

\begin{notn}
For two prestacks $\mcX, \mcY \in \PSt_{/S}$, we will use the following
notation:
\begin{figure}[H]
\begin{tabularx}{\textwidth}{l c l}
$p_\mcX$ & : & the structure $1$--morphism $\mcX \to \Aff_{/S}$ \\
$\Map_S(\mcX, \mcY)$ & : & the mapping prestack relative to $S$ \\
$\Gl_{n}$ & : & the general linear group scheme over $k$ of degree $n$ \\
$\Gl_{n, S}$ & : & the fibre product $\Gl_{n} \times_k S$ \\
$\sM$ & : & the prestack $\coprod_{n = 0}^\infty B\Gl_{n, S}$ \\
$\sM(\mcX)$ & : & the prestack
  $\Map_S(\mcX, \sM) \cong \coprod_{n = 0}^\infty \Map_S(\mcX, B\Gl_{n, S})$ \\
$\sM^n(\mcX)$ & : & $\Map_S(\mcX, B\Gl_{n, S})$ \\
$\sB(\mcX)$ & : & the prestack $\sM(\mcX) \times_S \mcX$ \\
$\sB^n(\mcX)$ & : & $\sM^n(\mcX) \times_S \mcX$ \\
\end{tabularx}
\end{figure}
\end{notn}

\begin{prop}\label{prop:prest-prod-strict}
For a diagram $\mcY \to[p] S \ot[q] \mcZ$ in $\PSt$,
the underlying category of the fibre product $\mcY \times_S \mcZ$ is strictly
isomorphic to the strict fibre product of categories $\mcY \times_S^{str} \mcZ$.
\end{prop}
\begin{proof}
By \cite[\href{https://stacks.math.columbia.edu/tag/0040}{Lemma 0040}]
{stacks-project},
an object of $\mcY \times_S \mcZ$ is a tuple $(U, y, z, f)$ where
$U \in \Ob{\Aff}$, $y \in \Ob{\mcY_U}, z \in \Ob{\mcZ_U}$,
$f$ is an isomorphism $p(y) \to[\sim] q(z)$ in $S_U$.
However, since the fibre category $S_U$ is discrete --- that is, the set
of morphisms of $k$--schemes $U \to S$ --- $f$ must be
the identity $\id_u$ for some $u : U \to S$ and, we must have
$p(y) = u = q(z)$.

The morphism sets $(\mcY \times_S \mcZ)((U, y, z, f), (U', y', z', f))$
consist of tuples $(a, b)$ where $a : y \to y', b : z \to z'$ are morphisms
in $\mcY, \mcZ$ respectively such that $p(a), q(b) : U \to U'$ are the same
morphism of $k$--schemes and the following diagram commutes:
\[\begin{tikzcd}
p(y) \ar[r, "p(a)"] \ar[d, "f" left] & p(y') \ar[d, "f'"] \\
q(z) \ar[r, "q(b)" below] & q(z')
\end{tikzcd}\]
which is equivalent to the first condition since $f = \id_U, f' = \id_{U'}$
by the previous paragraph. This gives an equality of sets:
\[
(\mcY \times_S \mcZ)((U, y, z, f), (U', y', z' f'))
= (\mcY \times_S^{str} \mcZ)((y, z), (y', z'))
\]

From this description, we can see that the mappings
$(U, y, z, f) \mapsto (y, z)$ and $(a, b) \mapsto (a, b)$ are bijections, and it is
straightforward to check that they assemble to a functor
$\mcY \times_S \mcZ \to \mcY \times_S^{str} \mcZ$, which must be
a strict isomorphism, as it induces bijections of object and morphism sets.
\end{proof}

\begin{prop}\label{prop:prest-prod-strict-nat}
In the context \cref{prop:prest-prod-strict}, if we have two $2$--morphisms in
$\PSt_{/S}$ $\alpha, \beta : f \To g : \mcW \to \mcY \times_S \mcZ$, then,
$\alpha = \beta$ if and only if $pr_i \star \alpha = pr_i \star \beta$ for both
$i = 1, 2$.
\end{prop}
\begin{proof}
We observe that for an object $w \in \mcW$,
$\alpha_w = (\alpha_{w, 1}, \alpha_{w, 2})$ for two morphisms
$\alpha_{w, i} : pr_i(f(w)) \to pr_i(g(w))$ for $i = 1, 2$ in $\mcY, \mcZ$
respectively. Similarly $\beta_w = (\beta_{w, 1}, \beta_{w, 2})$
for morphisms $\beta_{w, i} : pr_i(f(w)) \to pr_i(g(w)), i = 1, 2$.
Then, $\alpha_w = \beta_w$ if and only if
$(pr_i \star \alpha)_w = pr_i(\alpha_w) = \alpha_{w, i}
= \beta_{w, i} = pr_i(\beta_w) = (pr_i \star \beta)_w$.
\end{proof}

\subsection{Sheaves on Prestacks}

For every $\mcX \in \PSt_{/S}$, we will consider $\mcX$ equipped with
the \'etale topology inherited from $\Aff_{/S}$ in the sense of
\cite[\href{https://stacks.math.columbia.edu/tag/06NV}{Definition 06NV}]
{stacks-project}. This is generally referred to as the large \'etale site of
$\mcX$. By a sheaf on $\mcX$, we will mean a functor
$X^\op \to \Set$ that satisfies the sheaf condition with respect to this
inherited \'etale topology. By the $2$--Yoneda lemma, this site is equivalent
to the site $\Aff_{/S/\mcX}$ whose objects are morphisms of prestacks
$U \to \mcX$ over $S$, where $U$ is an affine scheme over $S$,
and coverings are \'etale coverings.

\begin{notn}
For any two prestacks $\mcX$, $\mcY, \mcZ$, any morphism of prestacks
$r = (f, g) : \mcY \to \mcZ^2$, and any presheaves $A, E, F$ on $\mcX$, which
may be (pre)sheaves of Abelian groups, rings or modules, depending
on context, we will use the following notation:
\begin{figure}[H]
\begin{tabularx}{\textwidth}{l c l}
$\mcO_{\mcX}$ & : & the structure sheaf of $\mcX$ defined by the composite
  $\mcX \to[p_\mcX] \Aff_{/S} \to[\Gamma] \Set$ \\
$E|_x$ & : & the composite $\mcX_{/x}^\op \to \mcX^\op \to[E] \Set$ for
  $x \in \Ob{\mcX}$ \\
$\LMod(A)$ & : & category of left $A$--modules \\
$\RMod(A)$ & : & category of right $A$--modules \\
$\Mod(A)$ & : & category of $A$--bimodules \\
$\QCoh(\mcX)$ & : & category of quasicoherent $\mcO_\mcX$--modules \\
$\Vect(\mcX)$ & : & category of finite locally free $\mcO_\mcX$--modules \\
$\Nat(E, F)$ & : & set of natural transformations $E \to F$ \\
$\HHom(E, F)$ & : & the internal $\Hom$ object in the topos $\PSh(\mcX)$ \\
$\Hom_{\mcO_\mcX}(E, F)$ & : & set of $\mcO_\mcX$--module maps $E \to F$ \\
$\Gamma(E)$ & : & alternate notation for $\Hom_{\mcO_\mcX}(\mcO_\mcX, E)$ \\
$\HHom_{\mcO_\mcX}(E, F)$ & : & the internal $\Hom$ object in the Abelian
  category $\Mod(\mcO_\mcX)$ \\
$[E, F]$ & : & notation for $\HHom_{\mcO_\mcX}(E, F)$ \\
$E^\vee$ & : & notation for $\HHom_{\mcO_\mcX}(E, \mcO_\mcX)$
\end{tabularx}
\end{figure}
\end{notn}

\begin{warn}
In many references, such as \cite{Champs-Alg}, \cite{AlperModuli} or
\cite{AG-Olsson}, sheaves on an algebraic stack $\mcX$, not general prestacks,
are defined to be sheaves on the full subcategory of $\Aff_{/S/\mcX}$ consisting
of the smooth morphisms $U \to \mcX$, equipped with the \'etale topology.
This is called the lisse-\'etale site of $\mcX$, and we will denote it as
$\mcX_{\text{lis-\'et}}$.
There is a functor $\Sh(\mcX) \to \Sh(\mcX_{\text{lis-\'et}})$ given by
pre-composing with the inclusion
$\mcX_{\text{lis-\'et}} \to \Aff_{/S/\mcX} \simeq \mcX$, and this
functor is not an equivalence in general, but it restricts to an equivalence
$\QCoh(\mcX_{\text{lis-\'et}}) \to \QCoh(\mcX)$
\cite[\href{https://stacks.math.columbia.edu/tag/07B1}{Lemma 07B1}]
{stacks-project}. Hence, when dealing with quasi-coherent
sheaves, most results and techniques are readily transferrable between the two
approaches but some care might be necessary.
\end{warn}

\begin{rmk}
We recall that for a morphism of prestacks $f : \mcX \to \mcY$, the pullback
functor $f^* : \PSh(\mcY) \to \PSh(\mcX)$ is simply the functor
$\Fun(\mcY^\op, \Set) \to \Fun(\mcX^\op, \Set)$ defined by precomposition with
$f$. The key fact is that, since we are dealing with the large \'etale site,
this precomposition functor restricts to a functor on sheaves
$f^* : \Sh(\mcY) \to \Sh(\mcX)$
\cite[\href{https://stacks.math.columbia.edu/tag/06TS}{Lemma 06TS}]
{stacks-project} --- no sheafification is required.
It is easy to check that this makes
the functor $f^* : \Sh(\mcY) \to \Sh(\mcX)$ preserve all limits and
colimits (set-theoretic considerations aside).
Furthermore, since $f^*\mcO_{\mcX} = \Gamma \circ p_\mcX \circ f
= \Gamma \circ p_\mcY = \mcO_{\mcY}$, $f^*$ also restricts to the usual pullback
of modules $f^* : \Mod(\mcO_\mcX) \to \Mod(\mcO_\mcY)$, which in turn preserves
quasicoherent modules.
\end{rmk}

We now record some basic facts about sheaves on prestacks that should be well
known, will be necessary for proving our results but are difficult to find
references for.

\begin{prop}\label{prop:pullback-nat-OX-alg}
Let $f, g : \mcY \to \mcX$ be $1$--morphisms and $\alpha : f \To g$ a
$2$--morphism in $\PSt_{/S}$, and $A : \mcY^\op \to \Set$ a presheaf of
$\mcO_\mcX$--modules. Then,
$A \star \alpha^\op : g^*A \To f^*A : \mcX^\op \to \Set$ is a morphism of
$\mcO_\mcY$--modules. If $A$ is an $\mcO_\mcX$--algebra, then
$A \star \alpha^\op$ is a map of $\mcO_\mcY$--algebras.
\end{prop}
\begin{proof}
For any object $y \in \mcY$,
we first recall that $\mcO_\mcX$ is the composite
$\mcX \to[p_Y] \Aff_{/S} \to[\Gamma] \Set$ so that
\[
f^*\mcO_\mcX(y) = \mcO_\mcX(f(y)) = \Gamma(p_\mcX(f(y))) = \Gamma(p_\mcY(y))
  = \mcO_\mcY(y)
\]
The same holds for $g^*\mcO_\mcX(y)$ so that it is also equal to
$\mcO_\mcY(y)$. Next, since $\alpha$ is a $2$--morphism of $\PSt_{/S}$,
by definition, we have $p_\mcX(\alpha_y) = \id_{p_\mcY(y)}$
so that
\[
(\mcO_\mcX \star \alpha^\op)_y
= \mcO_\mcX(\alpha_y^\op)
= \Gamma(p_\mcX(\alpha_y^\op))
= \Gamma(\id_{p_\mcY(y)}^\op)
= \id_{\Gamma(p_\mcY(y))}
= \id_{\mcO_\mcY(y)}
\]

Let $s : \mcO_\mcX \times A \to A$ be the structure map. Then, the component
$(f^*s)_y : f^*\mcO_\mcX(y) \times f^*A(y) \to f^*A(y)$ is the map
$s_{f(y)} : \mcO_\mcX(f(y)) \times A(f(y)) \to A(f(y))$, and similarly
for $(g^*s)_y$.
By the naturality of $s$, the following diagram commutes:
\[\begin{tikzcd}[row sep=large, column sep=huge]
\mcO_\mcX(g(y)) \times A(g(y))
  \ar[r, "s_{g(y)}"]
  \ar[d, "\mcO_\mcX(\alpha_y^\op) \times A(\alpha_y^\op)" left] &
A(g(y)) \ar[d, "A(\alpha^\op_y)"] \\
\mcO_\mcX(f(y)) \times A(f(y)) \ar[r, "s_{f(y)}" below] &
A(f(y))
\end{tikzcd}\]
However, by the previous paragraph,
$\mcO_\mcX(\alpha_y^\op) = \id_{\mcO_\mcY(y)}$, and, on the other hand,
$A(\alpha_y^\op) = (A \star \alpha^\op)_y$. Thus, the above diagram is:
\[\begin{tikzcd}[row sep=large, column sep=huge]
\mcO_\mcY(y) \times g^*A(y)
  \ar[r, "(g^*s)_y"]
  \ar[d, "\id_{\mcO_\mcY, y} \times (A \star \alpha^\op)_y" left] &
g^*A(y) \ar[d, "(A \star \alpha^\op)_y"] \\
\mcO_\mcY(y) \times f^*A(y) \ar[r, "(f^*s)_y" below] &
f^*A(y)
\end{tikzcd}\]
and its commutativity along with the observation that $A(\alpha_y^\op)$ is a
morphism of Abelian groups shows that $A \star \alpha^\op$ is morphism of
$\mcO_\mcY$--modules. Furthermore, when $A$ is, in addition, a sheaf
of commutative rings, $A(\alpha_y^\op)$ is a morphism of commutative rings
and, $A \star \alpha^\op$ is a morphism of $\mcO_\mcY$--algebras.
\end{proof}

\begin{cns}\label{cns:pullback-Hom}
For a morphism of prestacks $f : \mcY \to \mcZ$ and two
$E, F$ on $\mcZ$, we have a map of presheaves:
\[
\xi_f : f^*\HHom(E, F) \to[] \HHom(f^*E, f^*F)
\]
defined as follows. Let $y \in \Ob{\mcY}$ and
$s \in f^*\HHom(E, F)(y)$.
Then, $f^*\HHom(E, F)(y) = \Hom(E|_{f(y)}, F|_{f(y)})$ and $s$ is a natural
transformation of functors
$E|_{f(y)} \To F|_{f(y)} : \mcZ_{/f(y)} \to \Set$. This, then, yields
a horizontal composite $s \star f_{/y}$, where
$f_{/y} : \mcY_{/y} \to \mcZ_{/f(y)}$ is the functor induced by $f$.
This is shown below:
\[\begin{tikzcd}[row sep=small]
& &
\mcZ
    \ar[rd, "E" above right]
    \ar[dd, Rightarrow, "s", shorten=0.5em] & \\
\mcY_{/y}
    \ar[r, "f_{/y}"] &
\mcZ_{/f(y)}
    \ar[ru, "F_{f(y)}" above left]
    \ar[rd, "F_{f(y)}" below left] & &
\Set \\ & &
\mcZ
    \ar[ru, "F" below right] & \\
\end{tikzcd}\]
This is a map of presheaves $f^*E|_y \to f^*F|_y$ and we denote this by
$f^*_{/y}(s)$ --- noting that it is precisely the pullback of the morphism
$s$ under $f_{/y}$.
\end{cns}

\begin{prop}
In the context of \cref{cns:pullback-Hom}, if $E, F$ are $\mcO_\mcZ$--modules,
then $\xi_f$ restricts to a morphism of sheaves:
\[
\xi_f : f^*\HHom_{\mcO_\mcZ}(E, F) \to[] \HHom_{\mcO_\mcY}(f^*E, f^*F)
\]
\end{prop}
\begin{proof}
This follows from the fact that pullbacks of morphisms of module presheaves
are morphisms of module presheaves, so that $f^*_{/y}(s)$ is a morphism
of modules whenever $s$ is.
\end{proof}

\begin{prop}\label{prop:pullback-Hom}
In the context of \cref{cns:pullback-Hom}, if $E, F$ are $\mcO_\mcZ$--module
sheaves and if $E$ is finite locally free, then $\xi_f$ is an isomorphism:
\[
\xi_f : f^*\HHom_{\mcO_\mcZ}(E, F) \to[\cong] \HHom_{\mcO_\mcY}(f^*E, f^*F)
\]
\end{prop}
\begin{proof}
Let $y \in \Ob{\mcY}$. We wish to show that the component $\xi_{f, y}$
is an isomorphism.
Since $\xi_f$ is a morphism of sheaves, it suffices to produce a cover of
$\set{y_i \to y}_{i \in I}$ and show that
$\xi_{f, y_i}$ is an isomorphism for each $i \in I$.
For this, we first choose a covering $\set{c_i : z_i \to f(y)}_{i \in I}$
such that $t_i : E|_{z_i} \to[\cong] \mcO_{\mcZ}|_{z_i}^{\oplus n_i}$ for
some finite integer $n_i$. This means
$\set{p_\mcZ(c_i) : p_\mcZ(z_i) \to p_\mcZ(f(y))}_{i \in I}$
is a cover of $U := p_\mcZ(f(y)) = p_\mcY(y)$ in $\Aff_{/S}$.
Now, $y$ corresponds to a map $m_y : U \to \mcY$ satisfying $m_y(\id_U) = y$,
so that $f(y) = f(m_y(\id_U))$. Letting $U_i := p_\mcZ(z_i)$, we get objects
$y_i := f(m_y(p_\mcZ(c_i)(\id_{U_i})))$. Then,
$d_i := p_\mcZ(c_i) : U_i \to U$ is a morphism in $\Aff_{/S/U}$ yielding
morphisms $m_y(d_i) : y_i \to y$ in $\mcY$ over $d_i : U_i \to U$, which form a
cover of $U$. This shows that $m_y(d_i) : y_i \to y$ form a cover of $y$.

Now, $\xi_{f, y_i}$ is a function:
\[
\xi_{f, y_i} : \Hom_{\mcO_\mcZ|_{f(y_i)}}(E|_{f(y_i)}, F|_{f(y_i)}) \to
\Hom_{\mcO_\mcY|_{y_i}}(f^*E|_{y_i}, f^*F|_{y_i})
\]
By construction, $f(y_i) = z_i$. We then observe that we have strict equalities
\[
f^*(E|_{y_i}) = f^*_{/y_i}(E|_{f(y_i)}) = f^*_{/y_i}(E|_{z_i}),
f^*(F|_{y_i}) = f^*_{/y_i}(F|_{f(y_i)}) = f^*_{/y_i}(F|_{z_i})
\]
and noticing that $\xi_{f, y_i}$ is given by pullback of morphisms of
presheaves along $f_{/u_i}$, we have a commutative diagram as follow, by
the functoriality of $f_{/y_i}^*$:
\[\begin{tikzcd}
\Hom_{\mcO_\mcZ|_{z_i}}(E|_{z_i}, F|_{z_i})
    \ar[r, "\xi_{f, y_i} = f_{/y_i}^*"]
    \ar[d, "- \circ t_i^{-1}" left] &
\Hom_{\mcO_\mcY|_{y_i}}(f^*_{/y_i}(E|_{z_i}), f_{/y_i}^*(F|_{z_i}))
    \ar[d, "- \circ f^*_{/y_i}(t_i^{-1})"] \\
\Hom_{\mcO_\mcZ|_{z_i}}(\mcO_\mcZ|_{z_i}^{\oplus n_i}, F|_{z_i})
    \ar[r, "f_{/y_i}^*" below] &
\Hom_{\mcO_\mcY|_{y_i}}(f^*_{/y_i}(\mcO_\mcZ|_{z_i}^{\oplus n_i}),
    f_{/y_i}^*(F|_{z_i}))
\end{tikzcd}\]
where the vertical maps are isomorphisms.
Hence, it suffices to show that the bottom horizontal map is a bijection.

Let $e_1, \dots, e_{n_i}$ be the standard generators of the free module
$\mcO_{\mcZ}(z_i)$--modules
$\mcO_{\mcZ}|_{z_i}^{\oplus n_i}(\id_{z_i}) = \mcO_\mcZ(z_i)^{\oplus n_i}$.
Then, the mapping
\[
\gamma : \Hom_{\mcO_\mcZ|_{z_i}}(\mcO_\mcZ|_{z_i}^{\oplus n_i}, F|_{z_i})
    \to F(z_i)^{\oplus n_i} :
s \mapsto (s_{\id_{z_i}}(e_1), \dots, s_{\id_{z_i}}(e_{n_i}))
\]
can be verified to be an isomorphism of $\mcO_{\mcZ}(z_i)$--modules. To
see this, we can produce a map in the reverse direction as follows.
Let $(r_1, \dots, r_{n_i}) \in F(z_i)^{\oplus n_i}$. Then, this provides
a map $\mcO_\mcZ|_{z_i}^{\oplus n_i} \to F|_{z_i}$ as follows: for each
object $z \to z_i$ in $\mcZ_{/z_i}$, we send
$(e_1, \dots, e_n)|_z$ to $(r_1, \dots, r_n)|_z$.
It is not hard to verify that this is natural in $z$, and is an
inverse to $\gamma$. We have a similar isomorphism:
\[
\gamma' :
\Hom_{\mcO_\mcY|_{y_i}}(f^*_{/y_i}(\mcO_\mcZ|_{z_i}^{\oplus n_i}),
    f_{/y_i}^*(F|_{z_i})) \to
f^*_{/y_i}(F|_{z_i})(\id_{y_i})
\]
We notice that $f^*_{/y_i}(F|_{z_i})(\id_{y_i}) = F(f(y_i)) = F(z_i)$,
and $f^*_{/y_i}(\mcO_\mcZ|_{z_i})(\id_{y_i}) = \mcO_\mcZ(f(y_i))
= \mcO_\mcZ(z_i)$.
We then observe that
\[
(f_{/y_i}^*s)_{\id_{y_i}}
= (s \star f_{/y_i})_{\id_{y_i}}
= s_{f_{/y_i}^*(\id_{z_i})}
= s_{\id_{z_i}}
\]
Together, these show that the following diagram commutes:
\[\begin{tikzcd}
\Hom_{\mcO_\mcZ|_{z_i}}(\mcO_\mcZ|_{z_i}^{\oplus n_i}, F|_{z_i})
    \ar[rr, "f_{/y_i}^*" above]
    \ar[rd, "\gamma" below left] & &
\Hom_{\mcO_\mcY|_{y_i}}(f^*_{/y_i}(\mcO_\mcZ|_{z_i}^{\oplus n_i}),
        f_{/y_i}^*(F|_{z_i}))
    \ar[ld, "\gamma'" below right] \\ &
F(z_i)^{\oplus n_i} &
\end{tikzcd}\]
which shows that horizontal map is an isomorphism, as required.
\end{proof}


\section{Relative Spectrum for Prestacks}\label{sec:RelSpec}

We define a relative spectrum for sheaves of algebras over prestacks using the
same formula used for sheaves on the lisse-\'etale of an algebraic stack
in \cite[\S 10.2.1]{AG-Olsson}. This is the main tool used in our construction
and study of the relevant moduli stacks.

\begin{defn}[Relative Spectrum For Prestacks]
\label{defn:rel-spec}
For a prestack $\mcX \in \PSt_{/S}$, and any presheaf $A : \mcX^\op \to \Set$ of
$\mcO_\mcX$--algebras, we define a category $\ul\Spec_\mcX(A)$ as follows:
\begin{itemize}
\item Objects are tuples $(u, \sigma)$ where:
  \begin{itemize}
  \item $u : U \to \mcX$ is a morphism in $\PSt_{/S}$ with $U \in \Aff_{/S}$
  \item $\sigma : u^*A \to \mcO_U$ is a morphism of
    $\mcO_\mcX|_x$--algebras
  \end{itemize}
\item Morphisms $(u, s) \to (u', s')$ are morphisms
  $a : u(\id_U) \to u'(\id_{U'})$ in $\mcX$
  such that, if $f := p_\mcX(a)$ and the natural transformation
  $u' \circ f \To u$ corresponding to $a$ is $\tilde{a}$, the following diagram
  of $\mcO_U$--algebras commutes:
  \[\begin{tikzcd}
  (u' \circ f)^*A = f^*(u')^*A
    \ar[from=rr, "A \star \tilde{a}^\op" above]
    \ar[rd, "f^*s'" below left] & &
  u^*A \ar[ld, "s" below right] \\ &
  (u' \circ f)^*\mcO_\mcY = \mcO_U &
  \end{tikzcd}\]
\end{itemize}
This category comes equipped with a functor
$\pi_A : \ul\Spec_\mcX(A) \to \mcX$ defined by sending $(u, s)$ to $u(\id_U)$
and $a$ to $a$.
The category $\ul\Spec_\mcX(A)$ will be called the relative
spectrum of $A$ over $\mcX$ and $\pi_A$ will be called its projection.
We will suppress mention of $\mcX$ when it is clear from context, in which
case, we will simply write $\ul\Spec(A)$ for $\ul\Spec_\mcX(A)$.
When $A$ is clear from context, we will write $\pi$ instead of $\pi_A$
\end{defn}

\begin{warn}
$\ul\Spec(A)$ is originally defined only for the case where $\mcX$ is an
algebraic stack and $A$ is a sheaf of $\mcO_{\mcX_{\text{lis-\'et}}}$--modules
on the lisse-\'etale site of $\mcX$. Our definition should subsume this
by the equivalence of the categories of quasicoherent sheaves on the large
\'etale site and those on the lisse-\'etale site
\cite[\href{https://stacks.math.columbia.edu/tag/07B1}{Lemma 07B1}]
{stacks-project}, but we will not address this point in this work as
we will directly show that our construction has some of the same properties
that we will need.
\end{warn}

We proceed to show that the relative spectrum is the Grothendieck
construction of the following functor.

\begin{cns}
Given a prestack $\mcX \in \PSt_{/S}$, and two presheaves $A, B$ of
$\mcO_\mcX$--algebras, we define the internal Hom object
$\CCAlg_{\mcO_\mcX}(A, B)$ in the category of $\mcO_\mcX$--algebras as follows.
We set
\[
\CCAlg_{\mcO_\mcX}(A, B)(x) := \CAlg_{\mcO_\mcX|_x}(A|_x, B|_x)
\]
where the right hand side is the $\mcO_\mcX(x)$--algebra of
$\mcO_\mcX|_x$--algebra morphisms $A|_x \to B|_x$.

Given a morphism $\phi : x \to x'$ in $\mcX$, we have a functor
$\ol\phi : \mcX_{/x} \to \mcX_{x'}$ sending $f : y \to x$ to
$\phi \circ f : y \to x'$, and commuting with the forgetful functors
$F_x : \mcX_{/x} \to \mcX, F_{x'} : \mcX_{/x'} \to \mcX$.
Consider an element $s : A|_{x'} \to B|_{x'}$ of
$\CAlg_{\mcO_\mcX|_{x'}}(A|_{x'}, B|_{x'})$. We define
$\CCAlg_{\mcO_\mcX}(A, B)(\phi)$ to be the pullback $\ol\phi^*(s)$ which is
a morphism of $\mcO_{\mcX}|_{x}$--algebras. Concretely, it is the
horizontal composite $s \star \ol\phi$:
\[\begin{tikzcd}[column sep=huge]
(\mcX_{/x})^\op \ar[r, "\ol\phi"] &
(\mcX_{/x'})^\op
  \ar[r, bend left=3em, ""{name=U, below}, "A|_{x'} = A \circ F_{x'}"]
  \ar[r, bend right=3em, ""{name=D, above}, "B|_{x'} = B \circ F_{x}" below]
  \ar[from=U, to=D, Rightarrow, "s'" description] &
\Set
\end{tikzcd}\]
noting that $A|_{x'} \circ \ol\phi = A \circ F_x \circ \ol\phi = A|_x$
and similarly for $B|_x$. This gives a presheaf of $\mcO_\mcX$--algebras:
\[
\CCAlg_{\mcO_\mcX}(A, B) : \mcX^\op \to \Set
\]
\end{cns}

\begin{prop}\label{prop:rel-spec-Gr-cons}
In the context of \cref{defn:rel-spec}, $\pi : \ul\Spec(A) \to \mcX$
is the Grothendieck construction
$\int \CCAlg_{\mcO_\mcX}(A, \mcO_\mcX) \to \mcX$.
\end{prop}
\begin{proof}
We let $P$ denote $\CCAlg_{\mcO_\mcX}(A, \mcO_\mcX)$.
Let $(x, s)$ be an object of $\int P$ so that $x \in \Ob{\mcX}$,
$s \in P(x) = \CAlg_{\mcO_\mcX|_x}(A|_x, \mcO_\mcX|_x)$. That is,
$s : A|_x \to \mcO_\mcX|_x$ is a map of $\mcO_\mcX|_x$--algebras.
Let $U := p_\mcX(x)$, for any $v : V \to U \in \Aff_{/S}$, choose a Cartesian
lift $\wh{x}(v) : \tilde{x}(v) \to x$ in $\mcX$ of $v$ with target $x$.
By the universal property of Cartesian morphisms, for morphisms
$v : V \to U, v' : V' \to U$ and $g : V \to V'$ with $v' \circ g = v$,
we get a unique morphism $\wh{x}(g) : \tilde{x}(v) \to \tilde{x}(v')$
commuting with the maps $\wh{x}(v), \wh{x}(v')$. It is straightforward to verify
that the $\wh{x}(v)$ and $\wh{x}(g)$ assemble to a functor
$\wh{x} : \Aff_{/S/U} \to \mcX_{/x}$ over $\Aff/S$. This provides a map
of prestacks: $\tilde{x} := F_x \circ \wh{x} : U \to \mcX$, where
$F_x$ is the forgetful functor.
By pulling back, we get a map of $\mcO_U$--algebras
$\tilde{s} := \wh{x}^*s : \wh{x}^*A|_x = \tilde{x}^*A \to
\wh{x}^*\mcO_\mcX|_x = \mcO_U$, which is an object of $\ul\Spec(A)$ over $x$.

Now, consider a morphism $\phi : (x, s) \to (x', s')$ in $\int P$, so that
$\phi : x \to x'$ is a morphism in $\mcX$ such that $P(\phi)(s') = s$.
Let $U := p_\mcX(x), U' = p_\mcX(x'), f = p_\mcX(\phi) : U \to U'$.
We get maps $\tilde{x} : U \to \mcX, \tilde{x'} : U' \to \mcX$ as in the
previous paragraph, and $\phi$ gives a $2$--morphism of prestacks
$\tilde{\phi} : \tilde{x'} \circ f \To \tilde{x} : U \to \mcX$, by
the universal properties of the Cartesian lifts chosen.
We then have the following pasting of $2$--cells:
\begin{equation}
\begin{tikzcd}[column sep=huge]
U^\op
  \ar[dd, "f" left]
  \ar[rd, "\tilde{x}" above right, ""{below left, name={x}}]
  \ar[rrd, bend left, "A|_x", ""{name=AU, below left}] & & \\ &
|[alias=X]|\mcX^\op
  \ar[r, "\mcO_\mcX" description]
  \ar[from=AU, Rightarrow, "s", shorten=0.5em] &
\Set \\
(U')^\op
  \ar[ru, "\tilde{x'}" below right]
  \ar[rru, bend right, "A|_{x'}" below right, ""{name=AD, above left}]
  \ar[from=AD, to=X, Rightarrow, "s'", shorten=0.5em]
  \ar[to=x, Rightarrow, shorten=0.75em, "\tilde{\phi}"] & &
\end{tikzcd}
\end{equation}
which along with the hypothesis that $P(\phi)(s') = s$ shows the commutativity
of the diagram:
\begin{equation}\label{eqn:rel-spec-Gro-cons-1}
\begin{tikzcd}
\tilde{x}^*A
  \ar[rd, "\tilde{s}" below left]
  \ar[from=rr, "A \star \tilde{\phi}^\op" above] & &
f^*\tilde{x'}^*A \ar[ld, "\tilde{s'}" below right] \\
& \mcO_U &
\end{tikzcd}
\end{equation}
That is, $(f, \phi)$ is a morphism in $\ul\Spec(A)$.
The assignments:
\[
(x, s) \mapsto (\tilde{x}, \tilde{s}),
\phi \mapsto \phi
\]
assemble to a functor $\Psi_A : \int P \to \ul\Spec(A)$. To see this, we first
observe that for any two objects $(x, s), (x', s') \in \int P$ with
$\Psi_A(x, s) = (\tilde{x}, \tilde{s}),
\Psi_A(x', s') = (\tilde{x}, \tilde{s'})$, and $U = \dom(u), U' = \dom(u')$, we
have $x = \tilde{x}(\id_U), x' \cong \tilde{x'}(\id_{U'})$ by construction.
Then, by construction $\Hom_{\int P}((x, s), (x', s'))$ and
$\Hom_{\ul\Spec(A)}((\tilde{x}, \tilde{s}), (\tilde{x'}, \tilde{s'}))$ are both
subsets of $\Hom_\mcX(x, x')$ defined by the equivalent conditions
$P(\phi)(s') = s$ and \ref{eqn:rel-spec-Gro-cons-1}, and are hence
equal, while the mapping of morphisms induced by $\Psi_A$ is just the identity
on this subset. Thus, $\Psi_A$ respects composition and identities, and is, at
the same time, fully faithful.
This functor commutes with the maps to $\mcX$ by construction.

We will now see that this functor is essentially surjective.
Let $(u, \sigma)$ be an object of $\ul\Spec(A)$.
Then, we get an object
$x_u := u(\id_U) \in \Ob{\mcX}$ over $U$ and we have a functor
$p_{\mcX/x} : \mcX_{/x} \to \Aff_{/S/U} = U$ obtained by taking the slice of
$p_\mcX$ over $x$ --- it is defined by sending
$b : y \to x$ to $p_\mcX(b) : p_\mcX(y) \to U$ --- such that
$F_x = u \circ p_{\mcX/x}$. We then get a morphism of $\mcO_{\mcX}|_x$--algebras
$s_{u, \sigma} := p_{\mcX/x}^*\sigma = \sigma \star p_{\mcX/x} :
A|_x = p_{\mcX/x}^*u^*A \to \mcO_\mcX|_x = p_{\mcX/x}^*\mcO_U$.
Thus, $(x_u, s_{u, \sigma})$ is an object of $\int P$.
Now, $\tilde{x_u}(\id_U)$ is the domain of a Cartesian lift
$\wh{x_u}(\id_U) : \tilde{x_u}(\id_U) \to x_u$ over $\id_U : U \to U$,
which is an isomorphism since the fibre categories of $\mcX$ are groupoids.
However, since this was a choice of Cartesian lift to begin with, we could take
this choice to be $\id_{x_u}$. Then, letting
$\phi := \wh{x_u}(\id_U) = \id_{x_u} = \id_{u(\id_U)}$, we see from the
definition of the relative spectrum, that $\phi$ is the identity morphism
$\Psi_A(x_u, s_{u, \sigma}) \to (u, \sigma)$, showing that
$\Psi_A$ is strictly surjective objects.
\end{proof}

We will now focus on the properties of the relative spectrum that will be
necessary to prove our results.

\begin{prop}\label{prop:vec-bun-proj-affine}
In the context of \cref{defn:rel-spec}, $\pi_A$ is an affine morphism.
\end{prop}
\begin{proof}
Let $f : U \to \mcX$ be any morphism where $U \in \Aff_{/S}$.
We note that the $2$--fibre
product of prestacks is simply the $2$--fibre product in the slice $2$--category
$\Cat_{/\Aff_{/S}}$
\cite[\href{https://stacks.math.columbia.edu/tag/0041}{Lemma 0041}]
{stacks-project}, which is, in turn, just the $2$--fibre product in $\Cat$.
Since the Grothendieck construction is compatible with $2$--pullback in $\Cat$,
$\int A \circ f$ is equivalent as a category to the $2$--fibre product
of categories $\int A \times^{\Cat}_\mcX \Aff_{/S/U}$.
By \cref{prop:rel-spec-Gr-cons}, $\ul\Spec(A) = \int A$ and thus, the fibre
product of prestacks $\ul\Spec(A) \times_\mcX U = \int A \times_\mcX U$
is, in fact, $\int A \circ f$. We can then unwrap the definition of
$\int A \circ f$ to see that it is the usual relative spectrum defined for
schemes in
\cite[\href{https://stacks.math.columbia.edu/tag/01LQ}{Section 01LQ}]
{stacks-project}, which is an affine scheme over $U$ and hence an affine scheme
over $S$.
\end{proof}

\begin{prop}\label{prop:geom-vec-bun-Gro-cons}
In the context of \cref{defn:rel-spec}, for any finite locally free
$\mcO_\mcX$--module $M$, $\pi : \ul\Spec(\Sym(M^\vee)) \to \mcX$ is the
Grothendieck construction $\int M \to \mcX$ of the functor
$M : \mcX^\op \to \Set$.
\end{prop}
\begin{proof}
We notice that since pullbacks for large \'etale sites are exact, they commute
with direct sums. They also commute with tensor products and colimits. Hence,
they commute with taking the symmetric algebra of a sheaf of modules over the
structure sheaf. The finite locally free hypothesis ensures
$(M^\vee)^\vee \cong M$. Thus, we have:
\begin{align*}
     & \CAlg_{\mcO_\mcX|_x}(\Sym(M^\vee)|_x, \mcO_\mcX|_x) \\
\cong& \CAlg_{\mcO_\mcX|_x}(\Sym(M^\vee|_x), \mcO_\mcX|_x) \\
\cong& \Hom_{\mcO_\mcX|_x}(M^\vee|_x, \mcO_\mcX|_x) \\
\cong& (M^\vee)^\vee(x) \\
\cong& M(x)
\end{align*}
This along with \cref{prop:rel-spec-Gr-cons} now yields the result.
\end{proof}

\begin{notn}
For any $\mcX \in \PSt_{/S}$ and any $\mcO_\mcX$--module $M$, we denote
$\ul\Spec_\mcX(\Sym(M^\vee))$ by $\bV M$.
\end{notn}

\begin{cns}\label{cns:sec-to-morphism}
For any morphism of prestacks $f : \mcY \to \mcZ$ in $\PSt_{/S}$, any
$\mcO_\mcZ$--modules $E, F$, consider
a strict factorization of $f$ as $\mcY \to[t] \bV[E, F] \to[\pi] \mcZ$, that is
$f = \pi \circ t$.
For any object $y \in \mcY$,
$t(y) = (f(y), s(y))$ for some
$s(y) \in [E, F](f(y)) = \Hom_{\mcO_\mcX|_{f(y)}}(E|_{f(y)}, F|_{f(y)})$ by
\cref{prop:geom-vec-bun-Gro-cons}.
We then observe that $f^*E(y) = E(f(y)) = E|_{f(y)}(\id_y)$ and, similarly,
for $F$ and $\mcO_\mcX|_{f(y)}$ so that
$s(y)_{\id_y}$ is an $\mcO_\mcY(y) = \mcO_\mcX|_{f(y)}(\id_y)$--linear map
$f^*E(y) \to f^*F(y)$.
We denote the collection of morphisms $\set{s(y)_{\id_y}}_{y \in \Ob{\mcY}}$
as $\sigma(t)$.
\end{cns}

\begin{prop}\label{prop:sec-to-morphism}
In the context of \cref{cns:sec-to-morphism}, the collection $\sigma(t)$ constitutes
a natural transformation of functors and is hence a morphism of
$\mcO_\mcY$--modules.
\end{prop}
\begin{proof}
For any morphism $b : y \to y'$, the map
\[
f^*[E, F](y') = [E, F](f(y')) \to f^*[E, F](y) = [E, F](f(y))
\]
is given by sending a natural transformation
\[
a : E|_{f(y')} \To F|_{f(y')} : \mcZ_{/f(y')} \to \mcZ \to \Set
\]
to the horizontal composite $a \star (f(b) \circ -)$, where $f(b) \circ -$
is the functor $\mcZ_{/f(y)} \to \mcZ_{/f(y')}$ given by post-composing with the
morphism $b : y \to y'$. This horizontal composite is shown below:
\[\begin{tikzcd}[row sep=tiny]
& &
\mcZ
  \ar[rd, "E"]
  \ar[dd, Rightarrow, "a", shorten=0.5em] & \\
\mcZ_{/f(y)} \ar[r, "f(b) \circ -"] &
\mcZ_{/f(y')}
  \ar[ru, "F_{f(y')}" above left]
  \ar[rd, "F_{f(y')}" below left] & &
\Set \\ & &
\mcZ
  \ar[ru, "F" below right] & \\
\end{tikzcd}\]
Unwrapping the definition for horizontal composition, we see that for an object
$q : v \to f(y)$ in $\mcZ_{/f(y)}$, the component
$(s(y') \star (f(b) \circ -))_{q}$
is $s(y')_{f(b) \circ q}$. Viewing $f(b) : f(y) \to f(y')$ as a morphism
in $\mcZ_{/f(y')}$ from the object $f(b)$ to $\id_{f(y')}$, we have the
following commutative diagram by the naturality of $s(y')$:
\[\begin{tikzcd}[column sep=huge]
E|_{f(y')}(\id_{f(y')})
  \ar[r, "s(y')_{\id_{y'}}"]
  \ar[d, "E|_(f(y'))(f(b))" left] &
F|_{f(y')}(\id_{f(y')})
  \ar[d, "F|_{f(y')}(f(b))"] \\
E|_{f(y')}(f(b))
  \ar[r, "s(y')_{f(b)}" below] &
F|_{f(y')}(f(b))
\end{tikzcd}\]
We observe that $E|_{f(y')}(f(b)) = E(f(y))$, $F|_{f(y')}(f(b)) = F(f(b))$.
By the fact that $s(b) : (f(y), s(y)) \to (f(y'), s(y'))$ is a morphism
in $\int [E, F]$, we obtain that
$s(y')_{f(b)} = (s(y') \star (f(b) \circ -)) = s(y)_{\id_y}$. Thus, the above
commutative square becomes:
\[\begin{tikzcd}
E(f(y))
  \ar[r, "s(y')_{\id_{y'}}"]
  \ar[d, "E|_(f(y'))(f(b))" left] &
F(f(y'))
  \ar[d, "F|_{f(y')}(f(b))"] \\
E(f(y))
  \ar[r, "s(y)_{\id_{y}}" below] &
F(f(y))
\end{tikzcd}\]
Thus, the collection $\set{s(y)_{\id_{y}}}_{y \in \Ob{\mcY}}$
form a natural transformation of functors $f^*E \to f^*F$, which is
$\mcO_{\mcY}$--linear, by construction.
\end{proof}


\begin{prop}\label{prop:section-nat}
In the context of \cref{cns:sec-to-morphism},
for any two maps $t, t' : \mcY \to \bV[E, F]$ and any $2$--morphism of prestacks
$\beta : t \To t'$, the following diagram of presheaves commutes:
\[\begin{tikzcd}[column sep=huge]
(\pi \circ t)^*E \ar[r, "(E \circ \pi) \star \beta"] \ar[d, "\sigma(t)" left] &
(\pi \circ t')^*E \ar[d, "\sigma(t')"] \\
(\pi \circ t)^*F \ar[r, "(F \circ \pi) \star \beta" below] & (\pi \circ t')^*F
\end{tikzcd}\]
\end{prop}
\begin{proof}
We denote $x := \pi \circ t, x' = \pi \circ t'$.
Let $y \in \Ob{\mcY}$. Then, $t(y) = (x(y), s(y))$ and $t'(y) = (x'(y), s'(y))$.
We then have a morphism $\beta_y : (x(y), s(y)) \to (x'(y), s'(y))$
in $\bV[E, F]$ so that $\beta_y$ is a morphism $x(y) \to x(y')$ in $\mcX$
such that $s'(y)_{\beta_y} = s'(y) \star (\beta_y \circ -) = s(y)$.
By considering $\beta_y$ as morphism $\beta_y \to \id_{x'(y)}$ in $\mcX$,
we obtain the following commutative diagram:
\[\begin{tikzcd}[column sep=huge]
E(x'(y)) = E|_{x'(y)}(\id_{x'(y)})
  \ar[r, "s(y')_{\id_{x'(y)}}"]
  \ar[d, "E(\pi(\beta_y)) = E|_{x(y)}(\beta_y)" left] &
F|_{x'(y)}(\id_{x'(y)}) = F(x'(y))
  \ar[d, "F|_{x'(y)}(\beta_y) = F(\pi(\beta_y))"] \\
E(x(y)) = E|_{x'(y)}(\beta_y)
  \ar[r, "s'(y)_{\beta_y} = s(y)_{\id_{x(y)}}" below] &
F|_{x'(y)}(\beta_y) = F(x(y))
\end{tikzcd}\]
which is precisely the diagram whose commutativity we need to show for each such
$y$.
\end{proof}


\section{Arrow Bundles}\label{sec:ArrowBun}

We will now revisit the moduli stack of vector bundle triples or
``arrow bundles'' first constructed in \cite{ModQuivBun},
and give two concrete descriptions thereof that will be
used to prove the algebraicity of the moduli stacks of connection triples
in \cref{sec:ArrowBunFormalGrpd}.

\subsection{Moduli Prestack of Arrow Bundles}

We will define the moduli prestack of arrow bundles first in the intuitive
way --- that is, as the prestack associated to a functor that assigns
to a scheme $U$, the groupoid of vector bundles on $U \times \mcX$ and whose
action on morphisms $U \to U'$ is given by pullback. This is in contrast
to \cite{ModQuivBun}, where we take the more ``basis-free'' route of using
mapping stacks. The main point is that the two definitions give equivalent
prestacks, which we show in the next subsection. To this end, we first fix some
notation that will be used in the rest of the paper.

\begin{notn}
Let $\mcX, \mcY$ be prestacks over $S$ and $E, F$, any two
$\mcO_\mcY$--modules. We will, then, use the following notation:
\begin{figure}[H]
\begin{tabularx}{\textwidth}{l c l}
$[E, F]_\square$ & : & the $\mcO_{\mcY \times \mcY}$--module
  $\HHom_{\mcO_{\mcY}}(pr_1^*E, pr_2^*F)$ \\
$\mcE(\mcX)$ & : & the universal locally free sheaf on
  $\mcY = \sB(\mcX) = \sM(\mcX) \times \mcX$ \\
$\mcE^n(\mcX)$ & : & the restriction of $\mcE(\mcX)$ to
  $\mcY = \sB^n(\mcX) = \sM^n(\mcX) \times \mcX$ \\
$\mcV(\mcX)$ & : & the $\mcO_{\sB(\mcX) \times \sB(\mcX)}$--module
  $[\mcE(\mcX), \mcE(\mcX)]_\square$ \\
$\mcV^{n, m}(\mcX)$ & : & the $\mcO_{\sB^n \times \sB(\mcX)^m}$--module
  $[\mcE^n(\mcX), \mcE^m(\mcX)]_\square$ \\
$\sV(\mcX)$ & : & $\bV\mcV(\mcX) = \bV [\mcE(\mcX), \mcE(\mcX)]_\square$ \\
$\sV^{n, m}(\mcX)$ & : &
  $\bV\mcV^{n, m}(\mcX) = \bV [\mcE^n(\mcX), \mcE^m(\mcX)]_\square$  \\
$f^<$ & : & the pullback functor $(f \times \id_\mcX)^*$ \\
$f_>$ & : & the pushforward functor $(f \times \id_\mcX)_*$ \\
$r^\lhd$ & : & the pullback functor $(r \times \Delta_X)^*$ \\
$r_\rhd$ & : & the pushforward functor $(r \times \Delta_X)_*$
\end{tabularx}
\end{figure}
\end{notn}

\begin{cns}\label{cns:mod-prest-arr-bun}
Given any prestack $\mcX \in \PSt_{/S}$, we define a category
$\sM_1(\mcX)$ with the following data:
\begin{itemize}
\item Objects are tuples $(u, E, F, s)$, where:
  \begin{itemize}
  \item $u : U \to S$ is an object of $\Aff_{/S}$
  \item $E, F \in \Vect(U \times \mcX)$ are finite locally free
    $\mcO_\mcX$--modules on $U \times \mcX$
  \item $s : E \to F \in \Hom_{\mcO_{U \times \mcX}}(E, F)$ is a morphism of
    $\mcO_\mcX$--modules
  \end{itemize}
\item Morphisms $(u, E, F, s) \to (u', E', F', s')$ are tuples
$(f, a, b)$ where:
  \begin{itemize}
  \item $f : U = \dom(u) \to U' = \dom(u')$ is a morphism in $\Aff_{/S}$
  \item $a : f^<E' \to E, b : f^<F' \to F$ are isomorphisms of
  $\mcO_{U \times \mcX}$--modules making the following diagram commute in
  $\Vect(U \times \mcX)$:
  \[\begin{tikzcd}
  f^<E' \ar[r, "a"] \ar[d, "f^<s'" left] & E \ar[d, "s"] \\
  f^<F' \ar[r, "b" below] & F
  \end{tikzcd}\]
  \end{itemize}
\item Given two morphisms
$(u, E, F, s) \to[{(f, a, b)}] (u', E', F', s') \to[{(g, c, d)}]
(u'', E'', F'', s'')$, the composite is defined to be
$(g \circ f, f^<c \circ a, f^< \circ b)$ noticing that the following
diagram commutes in $\Vect(U \times \mcX)$:
  \[\begin{tikzcd}
  (g \circ f)^<E'' = f^<g^<E''
    \ar[r, "f^<c"]
    \ar[d, "{(g \circ f)^<s''} = f^<g^<s''" left] &
  f^<E' \ar[r, "a"] \ar[d, "f^<s'" left] &
  E \ar[d, "s"] \\
  (g \circ f)^<F'' = f^<g^<F'' \ar[r, "f^<d" below] &
  f^<F' \ar[r, "b" below] &
  F
  \end{tikzcd}\]
\item The identity morphism of an object $(u, E, F, s)$ is the tuple
$(\id_U, \id_E, \id_F)$.
\end{itemize}
\end{cns}

\begin{prop}
The category $\sM_1(\mcX)$ of \cref{cns:mod-prest-arr-bun} has a functor
to $\Aff_{/S}$ defined by the mapping:
\[\begin{array}{ccc}
(u, E, F, s) &\mapsto& (u : U \to S) \\
(f, a, b) &\mapsto& f : U \to U'
\end{array}\]
making it a prestack over $S$.
\end{prop}
\begin{proof}
That the mapping defines a functor is straightforward to check. That
this defines a prestack over $S$ follows from the observation that it
is the Grothendieck construction or unstraightening of the contravariant
functor:
\[\begin{array}{ccc}
U &\mapsto& \Fun(\Delta^1, \Vect(U \times \mcX))^\simeq \\
(U \to[f] U') &\mapsto& f^< \circ -
\end{array}\]
We omit the details.
\end{proof}

\begin{rmk}
If $\mcX$ is a stack, then so is $\sM_1(\mcX)$, but we will not show this
directly, although one could argue that this follows from the descent of
morphisms of sheaves.
Instead, we will produce an equivalence with another prestack
whose descent property is easier to show from definitions whenever $\mcX$ is
a stack. The reason for this
indirection is that this equivalence will be necessary in proving our
main results, so we address it first.
\end{rmk}

\begin{defn}[Moduli Prestack of Arrow Bundles]
We will call the prestack $\sM_1(\mcX)$ of \cref{cns:mod-prest-arr-bun},
the moduli prestack of arrow bundles on $\mcX$.
\end{defn}

\subsection{Definition via Mapping Stacks}

We now revisit the definition of the moduli of arrow bundles given in
\cite{ModQuivBun} and show that it gives a prestack equivalent to
\cref{cns:mod-prest-arr-bun}. To achieve this, we will give a concrete
description of this prestack that was not given before. There are two
main purposes of this definition. First, it is easier to show the algebraicity
of this stack using results about the algebraicity of mapping stacks. Second,
it can be easily generalized to the setting
of derived stacks, where we eventually wish to work,
without having to explicitly define the functor of points.
The work in this subsection should be considered as evidence for the soundness
of the definition as well as a proof of algebraicity of the slightly more
intuitive definition (\cref{cns:mod-prest-arr-bun}) via the functor of points.
At the same time, as we will discuss in \cref{subsec:NonAbHodge}
it proves some of the claims made in that paper with only sketches of proofs,
We begin by recalling the definition of mapping prestacks.

\begin{defn}[Mapping Prestacks {\cite[\S 3]{Wang-BunG}}]
Given two prestacks $\mcY, \mcZ \in \PSt_{/S}$, we define the category
$\Map_S(\mcY, \mcZ)$ as follows:
\begin{itemize}
\item Objects are tuples $(u, f)$ where:
  \begin{itemize}
  \item $u : U \to S$ is an object of $\Aff_{/S}$
  \item $f : U \times_S \mcY \to \mcZ$ is a $1$--morphism in $\PSt_{/S}$
  \end{itemize}
\item Morphisms $(u, f) \to (u', f')$ are tuples $(g, a)$ where:
  \begin{itemize}
  \item $g : U = \dom(u) \to U' = \dom(u')$ is a morphism in $\Aff_{/S}$
  \item $a : f' \circ (g \times \id_\mcY) \To f : U \times_S \mcY \to \mcZ$
    is a $2$--morphism in $\PSt_{/S}$
  \end{itemize}
\item Composition is given by pasting $2$--cells.
\item The identity morphism of $(u, f)$ is $(\id_u, \id_f)$.
\end{itemize}
This category has a functor to $\Aff_{/S}$ making it a prestack over $S$.
We call it the mapping prestack with domain $\mcY$ and codomain $\mcZ$.
\end{defn}

\begin{prop}
Mapping prestacks are internal Hom objects in the $2$--category $\PSt_{/S}$
with respect to the product $- \times_S -$.
\end{prop}
\begin{proof}
This is well-known and we omit the proof.
\end{proof}

\begin{cns}[Moduli Prestack of Arrow Bundles]
\label{cns:mod-prest-arr-sec}
For any prestack $\mcX \in \PSt_{/S}$, we consider the following maps:
\begin{itemize}
\item $\ol{\Delta_\mcX} : S \to \Map_S(\mcX, \mcX \times \mcX)$
  corresponding to the diagonal map $\Delta_\mcX : \mcX \to \mcX \times \mcX$
  which sends a map $U \to S$ to the composite
  $U \times \mcX \to[pr_2] \mcX \to[\Delta_\mcX] \mcX \times \mcX$
\item $- \circ p_\mcX : \sM(\mcX)^2 \simeq \Map_S(S, \sM(\mcX)^2)
  \to \Map_S(\mcX, \sM(\mcX)^2)$, which sends a morphism
  $(v, w) : U \to \sM(\mcX)^2$ to the composite
  $U \times \mcX \to[pr_1] U \to[{(v, w)}] \sM(\mcX)^2$
\end{itemize}
We then define the prestack $\sM_1'(\mcX)$ by the following fibre product
of prestacks:
\[\begin{tikzcd}[column sep=huge]
\sM_1'(\mcX)
  \ar[r] \ar[d] \ar[rd, phantom, "\lrcorner" very near start] &
\Map_S(\mcX, \sV(\mcX))
  \ar[d, "\pi_{\sV(\mcX)} \circ -"] \\
\sM(\mcX)^2 \simeq \sM(\mcX)^2 \times S
  \ar[r, "{(- \circ p_{\mcX}) \times \ol{\Delta_\mcX}}" below] &
\Map_S(\mcX, \sM(\mcX)^2) \times \Map_S(\mcX, \mcX^2) \simeq
\Map_S(\mcX, \sB(\mcX)^2)
\end{tikzcd}\]
where the right vertical map is induced by the bundle projection of
$\sV(\mcX)$.
\end{cns}

\begin{prop}\label{prop:M1'-concrete}
For any prestack $\mcX \in \PSt_{/S}$, the underlying category
of the prestack $\sM_1'(\mcX)$ of \cref{cns:mod-prest-arr-sec} admits the
following concrete description:
\begin{itemize}
\item Objects are tuples $(u, s, v, w, \phi)$, where:
  \begin{itemize}
  \item $u : U \to S$ is an object of $\Aff_{/S}$,
  \item $s : U \times \mcX \to \sV(\mcX)$ is a $1$--morphism in $\PSt_{/S}$,
  \item $v, w : U \to \sM(\mcX)$ are $1$--morphisms in $\PSt_{/S}$, and
  \item $\phi : p \circ s \To (v \times \id_\mcX, w \times \id_\mcX) :
    U \times \mcX \to \sB(\mcX)^2$ is a $2$--morphism in $\PSt_{/S}$
  \end{itemize}
\item Morphisms $(u, s, v, w, \phi) \to (u', s', v', w', \phi')$ are tuples
  $(f, a, b, c)$ where:
  \begin{itemize}
  \item $f : U = \dom(u) \to U' = \dom(u')$ is a morphism in $\Aff_{/S}$,
  \item $a : s' \circ (f \times \id_\mcX) \To s : U \times \mcX \to \sV(\mcX)$
    is a $2$--morphism in $\PSt_{/S}$, and
  \item $b : v' \circ f \To v, c : w' \circ f \To w : U \to \sM(\mcX)$
    are $2$--morphisms in $\PSt_{/S}$,
  \end{itemize}
  such that the following diagram of natural transformations commutes:
  \[\begin{tikzcd}[column sep=huge]
  \pi_{\sV(\mcX)} \circ s' \circ (f \times \id_\mcX)
    \ar[r, "\phi' \star (f \times \id_\mcX)"]
    \ar[dd, "\pi_{\sV(\mcX)} \star a" left] &
  (v' \times \id_\mcX, w' \times \id_\mcX) \circ (f \times \id_\mcX)
    \ar[d, equal] \\ &
  ((v' \circ f) \times \id_\mcX, (w' \circ f) \times \id_\mcX)
    \ar[d, "{(b \times e_\mcX, c \times e_\mcX)}"] \\
  \pi_{\sV(\mcX)} \circ s \ar[r, "\phi" below] &
  (v \times \id_\mcX, w \times \mcX)
  \end{tikzcd}\]
  where $e_\mcX$ is the identity $2$--morphism of $\id_\mcX$.
\end{itemize}
\end{prop}
\begin{proof}
We will unwrap the limit defining $\sM_1'(\mcX)$.
An object of $\Map_S(\mcX, \sV(\mcX))$ consists of a tuple
$(u, s)$, where $u : U \to S \in \Aff_{/S}$ and
$s : U \times \mcX \to \sV(\mcX)$. A  morphism $(u, s) \to (u', s')$
is a tuple $(f, a)$, where $f : U = \dom(u) \to U' = \dom(u') \in \Aff_{/S}$
and $a : s' \circ (f \times \id_\mcX) \To s$ is a $2$--morphism in $\PSt_{/S}$.
The projection to $\PSt_{/S}$ is given by
$(u, s) \mapsto u, (f, a) \mapsto f$.
The morphism of prestacks
\[
\pi_{\sV(\mcX)} \circ - :
    \Map_S(\mcX, \sV(\mcX)) \to \Map_S(\mcX, \sB(\mcX)^2)
\]
is given by
\[
(u, s) \mapsto (u, \pi_{\sV(\mcX)} \circ s),
(f, a) \mapsto (f \circ u, \pi_{\sV(\mcX)} \star a)
    = (u', \pi_{\sV(\mcX)} \star a)
\]

An object of
$\sM(\mcX)^2 \simeq \Map_S(S, \sM(\mcX)^2) \simeq \Map_S(S, \sM(\mcX))^2$
is a tuple $(u, v, w)$, where $u : U \to S \in \Aff_{/S}$ and
$(v, w) : U \times S \simeq U \to \sM(\mcX)^2$ is a $1$--morphism
in $\PSt_{/S}$ --- note that every morphism $U \to \sM(\mcX)^2$ is of this
form by \cref{prop:prest-prod-strict}. Similarly, a morphism
$(u, v, w) \to (u', v', w')$ is a tuple $(f, b, c)$ where
$f : U = \dom(U) \to U' = \dom(u')$ is a morphism in $\Aff_{/S}$,
$b : v' \circ f \To v, c : w' \circ f \To w : U \to \sM(\mcX)$
are $2$--morphisms in $\PSt_{/S}$. The structure map to $\Aff_{/S}$
is given by:
\[
(u, v, w) \mapsto u, (f, b, c) \mapsto f
\]
Now, consider the map
$- \circ p_\mcX : \Map_S(S, \sM(\mcX)) \to \Map_S(\mcX, \sM(\mcX))$ given
by precomposition with the structure map $p_\mcX : \mcX \to S$. We observe that
on an object $(u, v) \in \Map_S(S, \sM(\mcX))$, the map has value:
\[
(u, U \times \mcX
      \to[\id_U \times p_\mcX] U \times S
      \to[\simeq] U
      \to[v] \sM(\mcX))
\]
Then, $\id_U \times p_\mcX$ composed with the equivalence $U \times S \simeq S$
is just the projection $U \times \mcX \to U$. On a morphism
$(f, b) : (u, v) \to (u', v')$ in $\Map_S(S, \sM(\mcX))$, the map
$- \circ p_\mcX$ has value:
\[
(f : U \to U',
b \times e_\mcX
  : (v' \circ f) \times \id_\mcX \To v \times \id_\mcX
  : U \times \mcX \to \sB(\mcX))
\]
Then, using \cref{prop:prest-prod-strict}, the map
$- \circ p_\mcX : \Map_S(S, \sM(\mcX)^2) \to \Map_S(\mcX, \sM(\mcX))$
is given by:
\[\begin{array}{ccc}
(u, v, w) & \mapsto &
  (u, U \times \mcX \to[pr_1] U \to[v] \sB(\mcX),
      U \times \mcX \to[pr_2] U \to[w] \sB(\mcX)) \\
(u, b, c) & \mapsto & (u, b \times e_\mcX, c \times e_\mcX)
\end{array}\]

These descriptions along with the concrete description of fibre products
of prestacks given in
\cite[\href{https://stacks.math.columbia.edu/tag/0040}{Lemma 0040}]
{stacks-project} prove the result.
\end{proof}

\begin{rmk}\label{rmk:mod-prest-arr-sec}
An object of the underlying category of $\sM_1'(\mcX)$
consists of maps
$U \to S, s : U \times_S \mcX \to \sV(\mcX),
r = (v, w) : U \to \sM(\mcX)^2$, such that following diagrams $2$--commute:
\[\begin{tikzcd}
U \times_S \mcX \ar[r, "s"] \ar[d, "pr_1" left] & \sV(\mcX) \ar[d] \\
U \ar[r, "r" below] & \sM(\mcX) \times_S \sM(\mcX)
\end{tikzcd}\hspace{3.5em}
\begin{tikzcd}
U \times_S \mcX \ar[r, "s"] \ar[d, "pr_2" left] & \sV(\mcX) \ar[d] \\
\mcX \ar[r, "\Delta_\mcX" below] & \mcX \times_S \mcX
\end{tikzcd}\]
The $2$--commutativity above means that the map
$U \times_S \mcX \to[s] \sV(\mcX) \to[\pi_{\sV(\mcX)}] \sB(\mcX)^2$
is of the form $r \times \Delta_X = (v, w) \times \Delta_X$
which is the same as $(v \times_S \id_\mcX, w \times_S \id_\mcX)$ by the identification
$\Map_S(\mcX, \sB(X)^2) \simeq
\Map_S(\mcX, \sM(\mcX)^2) \times \Map_S(\mcX, \mcX^2)$.
That is, the object is a $2$--categorical diagram in $\PSt_{/S}$ of the form:
\[\begin{tikzcd}[row sep={huge}, column sep=huge]
& \sV(\mcX) \ar[d, "\pi_{\sV(\mcX)}"] \\
U \times \mcX
  \ar[r, bend left, ""{below, name=U},
    "\pi_{\sV(\mcX)} \circ s" above]
  \ar[r, bend right, ""{above, name=D},
    "{(v \times \id_\mcX, w \times \id_\mcX)}" below]
  \ar[ru, bend left, "s" above left]
  \ar[from=U, to=D, Rightarrow, "\phi"] &
\sB(X)^2
\end{tikzcd}\]
Roughly speaking, this data corresponds to a global section
of $(v, w)^\lhd[\mcE(\mcX), \mcE(\mcX)]_{\square}$, where
$pr_i : \sB(\mcX)^2 \to \sB(\mcX)$ are the projections onto the two factors
of $\sB(\mcX)$ and not $\sM(\mcX)$ or $\mcX$ separately.
Now, $(v, w)$ factors through
$\sM^n(\mcX) \times_S \mcX \times_S \sM^m(\mcX) \times_S \mcX$ for some $n, m$
so that $(v, w)^\lhd[\mcE(\mcX), \mcE(\mcX)]_{\square}$
is isomorphic to the following by \cref{prop:pullback-Hom}:
\begin{align*}
 & (v, w)^\lhd[pr_1^*\mcE(\mcX), pr_2^*\mcE(\mcX)] \\
=& [(v, w)^\lhd pr_1^*\mcE(\mcX), (v, w)^\lhd pr_2^*\mcE(\mcX)] \\
=& [v^< \mcE(\mcX), w^< \mcE(\mcX)] \\
=& [E, F]
\end{align*}
for two finite locally free sheaves
$E := v^<\mcE(\mcX), F := w^<\mcE(\mcX)$ on $U \times \mcX$, corresponding to
the maps $v, w$ respectively.
Thus, an object of $\sM_1'(\mcX)$ consists of
two finite locally free sheaves $E, F$ on $U \times \mcX$ and a global
section of $[E, F]$ which is a morphism of $\mcO_{U \times \mcX}$--modules
$E \to F$, much like $\sM_1(\mcX)$.
\end{rmk}

We will now see that this remark can be made precise into an
equivalence of prestacks $\sM_1'(\mcX) \simeq \sM_1(\mcX)$.

\begin{thm}\label{thm:arr-bun-arr-sec-equiv}
For any prestack $\mcX \in \PSt_{/S}$, we have an equivalence of prestacks
$\sM_1'(\mcX) \simeq \sM_1(\mcX)$.
\end{thm}
\begin{proof}
We will define a morphism of prestacks
$\Psi : \sM_1'(\mcX) \to \sM_1(\mcX)$ and show that it is fully faithful
and essentially surjective.
For an object $(u, s, v, w, \phi) \in \Ob{\sM(\mcX)}$, we obtain finite locally
free sheaves $E_0 = (pr_1 \circ \pi_{\sV(\mcX)} \circ s)^*\mcE(\mcX),
F_0 = (pr_2 \circ \pi_{\sV(\mcX)} \circ s)^*\mcE(\mcX),
E_1 := v^<\mcE(\mcX),
F_1 := w^<\mcE(\mcX)$ on
$U \times \mcX$, where $U = \dom(u)$. Then, \cref{prop:sec-to-morphism} gives us
a morphism of $\mcO_{U \times \mcX}$--modules $s_0 : E_0 \to F_0$. Next, $\phi$
gives us isomorphisms of $\mcO_{U \times \mcX}$--moduels
$\phi_1 := (pr_1 \circ \pi_{\sV(\mcX)}) \star \phi : E_0 \to E_1,
\phi_2 := (pr_2 \circ \pi_{\sV(\mcX)}) \star \phi : F_0 \to F_1$, hence a unique
morphism of $\mcO_{U \times \mcX}$--moduels $s_1 : E \to F$ making the following
diagram commute:
\[\begin{tikzcd}
E_0 \ar[r, "\phi_1"] \ar[d, "s_0" left] & E_1 \ar[d, "s_1"] \\
F_0 \ar[r, "\phi_2" below] & F_1
\end{tikzcd}\]
Thus, we have a triple $(E_1, F_1, s_1)$ constituting an object --- call it
$\Psi(u, s, v, w, \phi)$ --- of $\sM_1(\mcX)$.

Let $(f, a, b, c) : (u, s, v, w, \phi) \to (u', s', v', w', \phi')$
be a morphism of $\sM_1'(\mcX)$. We denote the sheaves and morphisms associate
to the codomain object as in the previous paragaph by $E_0', E_1', F_0', F_1',
s_0', s_1', \phi_0', \phi_1'$. We obtain, by pasting arrows, a cube as follows:
\begin{equation}\label{eqn:arr-bun-arr-sec-equiv-cube}
\begin{tikzcd}[row sep=huge, column sep=huge]
f^<E_0'
  \ar[rr, "f^<\phi_0'"]
  \ar[dd, "f^<s_0'" left]
  \ar[rd, "(pr_1 \circ \pi_{\sV(\mcX)}) \star a" description] & &
f^<E_1'
  \ar[rd, "b \times e_\mcX" description]
  \ar[dd, "f^<s_1'" near end] & \\ &
E_0
  \ar[dd, "s_0" left, near start]
  \ar[rr, crossing over, "\phi_0" near start] & &
E_1
  \ar[dd, "s_1"] \\
f^<F_0'
  \ar[rd, "(pr_2 \circ \pi_{\sV(\mcX)}) \star a" description]
  \ar[rr, "f^<\phi_1'" below, near end] & &
f^<F_1'
  \ar[rd, "c \times e_\mcX" description] & \\ &
F_0
  \ar[rr, "\phi_1" below] & &
F_1
\end{tikzcd}
\end{equation}
where all arrows, except possibly the vertical ones, are isomorphisms.
The front and back faces commute by the previous paragraph. The left face commutes
by \cref{prop:section-nat}. Noticing that
$(pr_i \circ \pi_{\sV(\mcX)}) \star (b \times e_\mcX, c \times \mcX) = b$
when $i = 1$ and $c$ when $i = 2$, the top and bottom faces commute by the
commutativity constraint in the description of morphisms of $\sM_1'(\mcX)$ given in
\cref{prop:M1'-concrete}. These imply that the right face commutes.
Hence, the map $f : U \to U'$ along with the right face of the above
cube provides a morphism
$(b, c) : (E_1, F_1, s_1) \to (E_1', F_1', s_1')$ in $\sM_1(\mcX)$
--- call it $\Psi(f, a, b, c)$.
It is straightforward to show that the assignment $\Psi$ constitutes a functor
$\sM_1'(\mcX) \to \sM_1(\mcX)$, which commutes with the projections to
$\Aff_{/S}$ by construction, and is, hence a morphism of prestacks.

The Gorthendieck construction gives an equivlance between categories fibred in
groupoids over $\Aff_{/S}$ and pseudofunctors $\Aff_{/S} \to \Grpd$. Therefore,
it suffices to check that the morphism
of prestacks is an equivalence when restricted to the fibre categories --- this
is the analogue, in the context of prestacks, of the pointwise criterion for
equivalence of pseudofunctors.
For any $U \to S$ in $\Aff_{/S}$, consider two morphisms
$(\id_U, a, b, c), (\id_U, \alpha, \beta, \gamma) : (u, s, v, w, \phi) \to
(u', s', v', w', \phi')$ that map to the same morphism under $\Psi$.
Then, we have an equality as follows, by hypothesis:
\begin{align*}
b \times e_\mcX &= \beta \times e_\mcX \iff b = \beta \\
c \times e_\mcX &= \gamma \times e_\mcX \iff c = \gamma
\end{align*}
By the observation that all morphisms in \cref{eqn:arr-bun-arr-sec-equiv-cube}
except the vertical ones are isomorphisms, we then have equalities:
\[\begin{array}{ccc}
(pr_1 \circ \pi_{\sV(\mcX)}) \star a &= b \times e_\mcX =& (pr_1 \circ \pi_{\sV(\mcX)}) \star \alpha \\
(pr_2 \circ \pi_{\sV(\mcX)}) \star a &= c \times e_\mcX =& (pr_2 \circ \pi_{\sV(\mcX)}) \star \alpha
\end{array}\]
Then, by \cref{prop:prest-prod-strict-nat}, we have $a = \alpha$. This shows
that $\Psi$ is fully faithful on fibre categories.

Consider an object $(E, F, f)$ of $\sM_1(\mcX)$ over $u : U \to S$.
Then, by the universal property
of the prestack $\sM(\mcX)$, we obtain two morphisms of prestacks
$v, w : U \to \sM(\mcX)$ and two isomorphisms of $\mcO_{U \times \mcX}$--modules
$\phi_E : v^<\mcE(\mcX) \to E, \phi_F : w^<\mcE(\mcX) \to F$. Thus, we have a
unique morphism $f' : v^<\mcE(\mcX) \to w^<\mcE(\mcX)$ of
$\mcO_{U \times \mcX}$--modules making the following diagram commute:
\[\begin{tikzcd}
v^<\mcE(\mcX) \ar[r, "\phi_E"] \ar[d, "f'" left] & E \ar[d, "f"] \\
w^<\mcE(\mcX) \ar[r, "\phi_F" below] & F
\end{tikzcd}\]
By \cref{prop:sec-to-morphism}, $f'$ gives a section of
$s'' : U \times \mcX \to \bV[v^<\mcE(\mcX), w^<\mcE(\mcX)]$.
However, by \cref{cns:pullback-Hom}, we have an isomorphism
\[
\xi_{(v \times \id_\mcX, w \times \id_\mcX)} :
(v, w)^\lhd[\mcE(\mcX), \mcE(\mcX)] \to[] [v^<\mcE(\mcX), w^<\mcE(\mcX)]
\]
This gives an equivalence of categories over $\mcX$:
\[
\int (\xi_{v \times \id_\mcX, w \times \id_\mcX)}):
\bV(v, w)^\lhd[\mcE(\mcX), \mcE(\mcX)] \to[] \bV[v^<\mcE(\mcX), w^<\mcE(\mcX)]
\]
which composed with $s''$ gives a section of
$s' : U \times \mcX \to \bV(v, w)^\lhd[\mcE(\mcX), \mcE(\mcX)]$.
In turn, since $\int \bV(v, w)^\lhd[\mcE(\mcX), \mcE(\mcX)]$
is the fibre product $(U \times \mcX) \times_{\sB(\mcX)^2} \sV(\mcX)$,
we have a canonical map
$\int \bV(v, w)^\lhd[\mcE(\mcX), \mcE(\mcX)] \to \sV(\mcX)$ which which by
composing with $s'$ gives a map $s : U \times \mcX \to \sV(\mcX)$. By
the universal property of $2$--fibre products, we get a $2$--morphism
\[
\phi : \pi_{\sV(\mcX)} \circ s \To (v \times \id_\mcX, w \times \id_\mcX)
: U \times \mcX \to \sB(\mcX)^2
\]
That is, $(u, s, v, w, \phi)$ is an object of $\sM_1'(\mcX)$. By construction,
under $\Psi$, this object maps to
$(v^<\mcE(\mcX), w^<\mcE(\mcX), f')$, which is isomorphic
to $(E, F, f)$ in $\sM_1(\mcX)$. Thus, $\Psi$ is essentially surjective on
fibre categories, as required.
\end{proof}

\begin{notn}
In light of the previous theorem, from this point onwards,
we will write $\sM_1(\mcX)$ for $\sM_1'(\mcX)$ as well, unless there is a need
to distinguish them.
\end{notn}

We now address the descent and algebaicity properties of the prestacks
constructed so far.

\begin{thm}
When $\mcX \in \PSt_{/S}$ is a stack, then so is $\sM_1(\mcX)$.
\end{thm}
\begin{proof}
If $\mcX$ is a stack, then so is $\sM(\mcX)$. Hence,
$\sB(\mcX)^2 = (\sM(\mcX) \times \mcX)^2$ is a stack.
Since $\sV(\mcX) \to \sB(X)^2$ is the Grothendieck construction
of a sheaf on $\sB(\mcX)^2$, it is a stack on the underlying category of
$\sB(\mcX)^2$, and from this, it can be shown that it satisfies the descent
property over $\Aff_{/S}$.
Thus, $\sM(\mcX)^2, \Map_S(\mcX, \sV(\mcX)), \Map_S(\mcX, \sB(\mcX)^2)$
are all stacks. Since the fibre product of a span of stacks is a stack,
$\sM_1(\mcX)$ is a stack by \cref{thm:arr-bun-arr-sec-equiv}.
\end{proof}

\begin{thm}\label{thm:mod-st-arr-bun-alg}
If $X$ is a scheme that is projective, flat and of finite presentation over $k$,
then $\sM_1(X)$:
\begin{enumerate}
\item is Artin,
\item is locally of finite presentation, and
\item has affine diagonal (in particular, is quasi-separated).
\end{enumerate}
\end{thm}
\begin{proof}[Proof sketch]
We observe that $\sM(X) = \coprod_{n = 0}^\infty \Map(X, B\Gl_n)$
is quasi-separated, locally of finite presentation, and has an affine diagonal by
\cite[Theorem 1.0.1.]{Wang-BunG}. The projection
$\sV(\mcX) \to \sB(\mcX)^2$ being affine allows us to conclude
that $\sV(\mcX)$ is a quasi-separated algebraic stack with affine stabilizers and
locally of finite presentation.
As a result, \cite[Theorem 1.2]{HR19} applies, and we can conclude that the
mapping stacks involved in \cref{cns:mod-prest-arr-sec} are algebraic, locally
of finite presentation and has affine diagonal. This shows that
$\sM_1'(X)$ is algebraic and hence \cref{thm:arr-bun-arr-sec-equiv} yields
the result. For a more elaborate argument, see \cite[\S 4]{ModQuivBun}.
\end{proof}

\begin{rmk}
We note that this algebraicity result holds much more generally. For this, we
refer to our result \cite[Theorem 4.10]{ModQuivBun}.
\end{rmk}

\subsection{The Boundary Map}

An important component of the proof of algebraicity of the moduli stack
of connection triples will will be the map $\sM'_1(\mcX) \to \sM(\mcX)^2$
that sends a morphism of vector bundles over $\mcX$ to the pair consisting
of the source and target of the morphism. We will record some properties
of this map for use in \cref{sec:ArrowBunFormalGrpd}. This will be called the
``boundary map'' as it sends a $1$--simplex in the category vector bundles
to the pair of vertices forming its boundary.

\begin{cns}\label{defn:bndry-map}
We define a morphism of prestacks $\partial_\mcX : \sM_1(\mcX) \to \sM(\mcX)^2$
as follows.
For an object $(u, E, F, s)$ of $\sM_1(\mcX)$ over $u : U \to S$, we define
\[
\partial_\mcX(u, E, F, s) := ((u, E), (u, F))
\]
noting that $E, F$ are finite locally free sheaves on $U \times \mcX$.
For a morphism $(f, a, b) : (u, E, F, s) \to (u', E', F', s')$, we define:
\[
\partial_\mcX(f, a, b) := ((f, a), (f, b))
\]
noting that $a : f^<E' \to E, b : f^<F' \to F$ are morphisms
in $\sM(\mcX)$ over $f$.
\end{cns}

\begin{prop}
For a prestack $\mcX \in \PSt_{/S}$,
the map $\partial_\mcX$ of \cref{defn:bndry-map} is a morphism of prestacks.
Let $\Psi$ again be the equivalence
$\sM_1'(\mcX) \to[\simeq] \sM_1(\mcX)$ of \cref{thm:arr-bun-arr-sec-equiv}.
Let $\partial_\mcX' : \sM_1'(\mcX) \to \sM(\mcX)^2$ be the projection of
of the fibre product defining $\sM_1'(\mcX)$.
Then, $\partial_\mcX \circ \Psi = \partial_\mcX'$ as $1$--morphisms.
\end{prop}
\begin{proof}
That $\partial_\mcX$ is a functor and commutes with the projections
to $\Aff_{/S}$ is immediate. For the second claim, we first notice that
the map $\partial_\mcX' : \sM_1(\mcX) \to \sM(\mcX)^2$ factors through
the equivalence
$\Map_S(S, \sM(\mcX))^2 \to[\simeq] \Map_S(S, \sM(\mcX)^2) \to[\simeq]
\sM(\mcX)^2$ sending
an object $(v, w) : U \simeq U \times S \to \sM(\mcX)^2$ to
the object $(v(\id_{U}), w(\id_U))$. The fibre product projection
$\sM_1(\mcX) \to \Map_S(S, \sM_1(\mcX))^2$ sends an object
$(u, s, v, w, \phi)$ to $(v, w)$. Hence, the map
$\sM_1(\mcX) \to \sM(\mcX)^2$ is given by sending $(v, w)$ to
$(v(\id_U), w(\id_U)) = (v^<\mcE(\mcX), w^<\mcE(\mcX))$. From the construction
of the map $\Psi$, we see that this is precisely
$\partial_\mcX(\Psi(u, s, v, w, \phi))$.

Now, let $(f, a, b, c) : (u, s, v, w, \phi) \to (u', s', v', w', \phi')$
be a morphism in $\sM_1'(\mcX)$. $\Psi$ sends this to
$(b \times e_\mcX, c \times e_\mcX)$, which is then send by $\partial_\mcX$
to itself, noting that the morphisms in $\sM_1(\mcX)$ are morphisms
in $\sM(\mcX)$ satisfying a commutativity constraint. Then, $\partial_\mcX'$
sends it to the $2$--morphism $(b, c) : (v, w) \To (v', w') : U \to \sM(\mcX)^2$.
The map $\Map_S(S, \sM(\mcX))^2 \to[\simeq] \sM(\mcX)^2$ sends it to
$(b \times e_\mcX, c \times e_\mcX)$. This shows that
$\partial_\mcX(\Psi(f, a, b, c)) = \partial_\mcX'(f, a, b, c)$, as required.

\end{proof}

\begin{prop}\label{prop:bndry-map-fully-faithful}
The map $\partial_\mcX$ of \cref{defn:bndry-map} is a faithful functor.
\end{prop}
\begin{proof}
We begin by observing that the Hom sets of $\sM_1(\mcX)$ are subsets
of the corresponding Hom sets of $\sM(\mcX)^2$ consisting of elements
satisfying a commutativity constraint. Also, the functions of Hom sets induced
by the morphism $\partial_\mcX : \sM_1(\mcX) \to \sM(\mcX)^2$ are the
inclusions of these subsets. Hence, $\partial_\mcX$ is faithful.
\end{proof}

\begin{defn}\label{defn:arr-bun-fixed-src-trgt}
For any prestack $\mcY \in \PSt_{/S}$, any affine scheme $U \in \Aff_{/S}$ and any
morphism $(v, w) : U \to \sM(\mcY)^2$, we denote by $P^{\mcY}_{v, w}$,
the following fibre product:
\[\begin{tikzcd}
P^\mcY_{v, w} \ar[r] \ar[d] \ar[rd, phantom, "\lrcorner" very near start] &
\sM_1(\mcY) \ar[d] \\
U \ar[r, "{(v, w)}" below] &
\sM(\mcY)^2
\end{tikzcd}\]
where the right vertical map is projection of the fibre product defining
$\sM_1(\mcX)$, in light of \cref{thm:arr-bun-arr-sec-equiv}.
\end{defn}

\begin{prop}\label{prop:bndry-map-fibre-concrete}
In the context of \cref{defn:arr-bun-fixed-src-trgt},
if we denote $E := v^<\mcE(\mcX), F := w^<\mcE(\mcX)$, the underlying category of
$P^\mcY_{v, w}$ has the following concrete description:
\begin{itemize}
\item Objects are tuples $(r, H, K, s, \phi_{H}, \phi_{K})$ where:
  \begin{itemize}
  \item $r : V \to U$ is a morphism in $\Aff_{/S}$
  \item $H, K \in \Vect(V \times \mcX)$
  \item $s : H \to K$ is a morphism of locally free sheaves over $V \times \mcX$
  \item $\phi_{H} : H \to r^<E, \phi_{K} : K \to r^<F$ are isomorphisms of
    locally free sheaves over $V \times \mcX$
  \end{itemize}
\item Morphisms $(r, H, K, s, \phi_{H}, \phi_{K}) \to
  (r', H', K', s', \phi_{H'}, \phi_{K'})$ are tuples
  $(f, a, b)$ where:
  \begin{itemize}
  \item $f : V = \dom(r) \to V' = \dom(r')$ is a morphism in $\Aff_{/S/U}$
  \item $a : f^<H' \to H, b : f^<K' \to K$ are isomorphisms of locally free sheaves
    over $V \times \mcX$ making the following diagrams commute:
    \[\begin{tikzcd}
    f^<H' \ar[r, "a"] \ar[d, "f^<s'" left] & H \ar[d, "s"] \\
    f^<K' \ar[r, "b" below] & K
    \end{tikzcd}\hspace{1.5em}
    \begin{tikzcd}
    f^<H' \ar[r, "a"] \ar[d, "f^<\phi_{H'}" left] & H \ar[d, "\phi_H"] \\
    f^<(r')^<E = r^<E \ar[r, "\id_{r^<E}" below] & r^<E
    \end{tikzcd}\hspace{1.5em}
    \begin{tikzcd}
    f^<K' \ar[r, "b"] \ar[d, "f^<\phi_{K'}" left] & K \ar[d, "\phi_K"] \\
    f^<(r')^<F = r^<F \ar[r, "\id_{r^<F}" below] & r^<F
    \end{tikzcd}\]
  \end{itemize}
\end{itemize}
\end{prop}
\begin{proof}
This follows from the definition of $\sM_1(\mcX)$ given in
\cref{cns:mod-prest-arr-bun}, the observation that the map $(v, w)$
has the values $(r^<E, r^<F)$ on objects $r : V \to U$ of $\Aff_{/S/U}$ and
$(\id_{(r' \circ f)^<E}, \id_{(r' \circ f)^<F})$ on morphisms $f : V \to V'$ in
$\Aff_{/S/U}$, and then computing the explicit description of the fibre product
\cite[\href{https://stacks.math.columbia.edu/tag/0040}{Lemma 0040}]
{stacks-project}.
\end{proof}

\begin{prop}\label{prop:bndry-map-fibre-simple}
In the context of \cref{defn:arr-bun-fixed-src-trgt}, let
$E = v^<\mcE(\mcX), F := w^<\mcE(\mcX)$, and consider the category
$Q^\mcY_{v, w}$ defined by the following data:
\begin{itemize}
\item Objects are tuples $(r, s)$ where:
  \begin{itemize}
  \item $r : V \to U$ is a morphism in $\Aff_{/S}$
  \item $s : r^<E \to r^<F$ is a morphism of finite locally free sheaves on
    $V \times \mcX$
  \end{itemize}
\item Morphisms $(r, s) \to (r' s')$ are morphisms
    $f : V = \dom(r) \to V' = \dom(r')$ in $\Aff_{/S/U}$ such that
    $r' \circ f = r$ and $f^<s' = s$.
\end{itemize}
We have an equivalence of prestacks
$\Psi_{v, w}^\mcY : Q^\mcY_{v, w} \to P^\mcY_{v, w}$.
\end{prop}
\begin{proof}
We observe that for any object $(r, s) \in Q^\mcY_{v, w}$, we have the object
\[
\Psi^\mcY_{v, w}(r, s) := (r, r^<E, r^<F, s, \id_{r^<E}, \id_{r^<F})
\]
in $P^\mcY_{v, w}$. For two objects $(r, s), (r', s')$, consider
a morphism $(f, a, b) : \Psi^\mcY_{v, w}(r, s) \to \Psi^\mcY_{v, w}(r', s')$.
Then, the commutativity constraint in the description of morphisms
in \cref{prop:bndry-map-fibre-concrete}, we have the following commutative
diagrams:
\[\begin{tikzcd}
f^<(r')^<E \ar[r, "a"] \ar[d, "f^<s'" left] & r^<E \ar[d, "s"] \\
f^<(r')^<F \ar[r, "b" below] & r^<F
\end{tikzcd}\hspace{1.5em}
\begin{tikzcd}
f^<(r')^<E  = r^<E\ar[r, "a"] \ar[d, "f^<(\id_{(r)'^<E}) = \id_{r^<E}" left] &
r^<E \ar[d, "\id_{r^<E}"] \\
f^<(r')^<E = r^<E \ar[r, "\id_{r^<E}" below] &
r^<E
\end{tikzcd}\]\[
\begin{tikzcd}
f^<(r')^<F  = r^<F\ar[r, "b"] \ar[d, "f^<(\id_{(r)'^<F}) = \id_{r^<F}" left] &
r^<F \ar[d, "\id_{r^<F}"] \\
f^<(r')^<F = r^<F \ar[r, "\id_{r^<F}" below] &
r^<F
\end{tikzcd}\]
showing that $a = \id_{r^<E}$, $b= \id_{r^<F}$, and
that $f^<s' = s$ so that $f$ is a morphism in $Q^\mcY_{v, w}$. Hence, we can
set $\Psi^\mcY_{v, w}(f) := (f, \id_{r^<E}, \id_{r^<F})$. It is not hard to show
that $\Psi^\mcY_{v, w}$ is a morphism of prestacks. That it is fully faithful is
evident from the construction. Suppose, now, that we have an object
$q = (r, H, K, s, \phi_H, \phi_K)$ of $P^\mcY_{v, w}$. This gives us a unique
map $t : r^<E \to r^<F$ making the following diagram commute:
\[\begin{tikzcd}
H \ar[r, "\phi_H"] \ar[d, "s" left] & r^<E \ar[d, "t"] \\
K \ar[r, "\phi_K" below] & r^<F
\end{tikzcd}\]
It is, then, easy to see that $(\id_{\dom(r)}, \phi_H, \phi_K)$ is an
isomorphism of the object $q$ with $\Psi^\mcY_{v, w}(r, t)$. This shows that
$\Psi^\mcY_{v, w}$ is essentially surjective.
\end{proof}

\begin{notn}
In light of \cref{prop:bndry-map-fibre-simple},
when the maps $v, w$ are not important, we will sometimes write
$Q^\mcY_{E, F} \simeq P^\mcY_{E, F}$.
\end{notn}

\begin{prop}\label{prop:bndry-map-repr-alg-sp}
In the context of \cref{defn:arr-bun-fixed-src-trgt}, if
$\sM_1(\mcX)$ and $\sM(\mcX)$ are algebraic, then $P^\mcY_{v, w}$
is an algebraic space.
\end{prop}
\begin{proof}
This follows from \cref{prop:bndry-map-fully-faithful} and
\cite[\href{https://stacks.math.columbia.edu/tag/04Y5}{Lemma 04Y5}]
{stacks-project}.
\end{proof}

\begin{warn}
In \cref{prop:bndry-map-repr-alg-sp}, the algebraicity of $\sM_1(\mcX)$
is a requirement and not a consequence.
\end{warn}

We will now see that in the case that $\mcY$ is a scheme $X$ flat, finitely
presented and projective over $S$, the boundary map is ``well-behaved''.

\begin{thm}\label{thm:X-fppf-proj-bndry-map-aff-fp}
Let $X \to S$ be a finitely presented, flat and projective morphism of schemes.
Then, $\partial_X : \sM_1(X) \to \sM(X)^2$ is affine and of finite presentation.
\end{thm}
\begin{proof}
We have to show that for any morphism $(v, w) : U \to \sM(X)^2$ where
$U \in \Aff_{/S}$, the map $P^X_{v, w} \to U$ is affine and of finite
presentation.
We obtain the following pasting of Cartesian squares by
\cref{thm:arr-bun-arr-sec-equiv} and the definition
(\ref{defn:arr-bun-fixed-src-trgt}) of $P^X_{v, w}$:
\[\begin{tikzcd}
P^X_{v, w} \ar[r] \ar[d] \ar[rd, phantom, "\lrcorner" very near start] &
\sM_1(X) \ar[r] \ar[d] \ar[rd, phantom, "\lrcorner" very near start] &
\Map_S(X, \sV(X)) \ar[d] \\
U \ar[r, "{(v, w)}" below] &
\sM(X)^2 \ar[r] &
\Map_S(X, \sB(X)^2)
\end{tikzcd}\]
The same argument as in the proof of theorem \cite[Lemma 3.2.1]{Wang-BunG}
shows that $P^X_{v, w}$ is equivalent to the stack
$\Map_U(U \times X, \sV(X) \times_{\sB(X)^2} (U \times X))$. However,
by \cref{prop:geom-vec-bun-Gro-cons}, $\sV(X) \times_{\sB(X)^2} (U \times X)$
is equivalent to $\int [E, F]$, where $E = v^<\mcE(X), F = w^<\mcE(X)$.
Furthermore, since $[E, F]$ is a finite locally free $\mcO_{U \times X}$--module
$\pi_A : \int [E, F] \to U \times X$ is affine by
\cref{prop:vec-bun-proj-affine}. By passing to charts, we can also show that
$\pi_A$ is of finite presentation. Finally, the projection
$pr_1 : U \times X \to U$ is finitely presented, flat and projective,
as it is a base change of the finitely presented, flat and projective morphism
$X \to S$. Lemma 3.1.4 of \cite{Wang-BunG} now applies and we can conclude that
the map $P^X_{v, w} \simeq \Map_U(U \times X, \int [E, F]) \to U$ is an affine
morphism of finite presentation, as required.
\end{proof}


\section{Arrow Bundles on Formal Groupoids}\label{sec:ArrowBunFormalGrpd}

We now have all the machinery we need to construct and study the moduli
stack of connection triples. We will address this in some generality: we will
take the approach of \cites{NonAbHodgeFilt} in constructing the stack
representing the non-Abelian Hodge filtration. Just like the moduli stack
of vector bundles on certain formal groupoids recover the moduli stacks
of connections, Higgs bundles, etc., we will consider
the moduli stack of vector bundle triples on those same formal groupoids
and show that they recover the desired moduli stacks of morphisms
of the respective objects: namely, $\lambda$--connections for a varying
parameter $\lambda$, usual connections (or, $1$--connections) and
Higgs bundles (or, $0$--connections).

\subsection{Review of Formal Groupoids}

We begin by reviewing formal groupoids both for the convenience of the reader
and to illustrate the difficulty in proving the algebraicity of the moduli
stack of connection triples using the same approach for the moduli stack
of vector bundle triples.

\begin{notn}
We will write $\FSch$ for the category of formal schemes over $k$.
\end{notn}

\begin{defn}[Formal Groupoid]
\label{defn:form-cat}
A formal category over $S$ is a tuple $(X, \sF, s, t, c, i)$ forming an
internal category in $\FSch_{/S}$, and satisfying:
\begin{enumerate}
\item $X$ is a scheme in $\Sch_{/S}$,
\item $i : X \to \sF$ is a closed immersion realizing $X$ as the underlying
scheme of $\sF$.
\end{enumerate}
This data is called a formal groupoid if, in addition:
\begin{enumerate}[resume]
\item for each $U \in \Sch_{/S}$, $(X(U), \sF(U), s_U, t_U, c_U, i_U)$
is a groupoid (internal to $\Set$).
\end{enumerate}
A formal category as above is said to be smooth if
\begin{enumerate}[resume]
\item the structure map $X \to S$ is smooth,
\item the morphisms $s, t : \sF \to X$ are formally smooth.
\end{enumerate}
\end{defn}

\begin{notn}
We will write a formal category as above simply as $(X, \sF)$,
when the structure maps are clear from context.
\end{notn}

For convenience, we make the following definition:
\begin{defn}[Formal Stack]
\label{defn:form-st}
A formal stack is a stack $\mathcal{X} \in \St_{/S}$ such that there exists
a formal groupoid $(X, \sF)$ and a $2$--coequalizer diagram in $\St_{/S}$:
\[\begin{tikzcd}
\sF \ar[r, shift left, "s"] \ar[r, shift right, "t" below] &
X \ar[r] &
\mathcal{X}
\end{tikzcd}\]
presenting $\mathcal{X}$ as a quotient stack.
We call $\mathcal{X}$ the stack associated to the formal groupoid
$(X, \sF)$.
\end{defn}

\begin{warn}
A formal stack is not necessarily a formal algebraic stack as defined in
\cite[Definition 5.3]{FormalAlgSt}. For instance, it may have a digonal
not representable by algebraic spaces --- see \cref{prop:Hodge-st-diag-non-rep}
--- which contradicts \cite[Lemma 5.12.]{FormalAlgSt}.
On the other hand, a formal algebraic stack is not necessarily a formal stack
since it may be a quotient of formal algebraic spaces more general than formal
schemes \cite[47]{FormalAlgSt}.
\end{warn}

\begin{notn}
In the context of the above definition, if the formal groupoid $(X, \sF)$
is clear from context, then we write $\mathcal{X}$ as $X_\sF$.
\end{notn}

There are three main examples of interest to us \cite[31---33]{NonAbHodgeFilt}.
The rough idea behind all of these is that a quasicoherent sheaf on
a formal stack is $(X, \sF)$ is a quasicoherent sheaf on $X$ along with
isomorphisms between the stalks encoded by $\sF$. By varying $\sF$,
and consequently the isomorphisms of stalks, we can recover connections and
Higgs bundles.

\begin{exm}[de Rham Stack]
\label{exm:dR-st}
If $X \to S$ is separated, then the diagonal $\Delta_{X/S} : X \to X \times_S X$
is a closed immersion and we take $\sF \to[{(s, t)}] X \times_S X$ to be the
formal completion of $X \times_S X$ along the set theoretic image
$\Delta_{X/S}(X)$. The composition morphisms $c : \sF \times_X \sF \to \sF$
is the one induced by the map
$(X \times_S X) \times_X (X \times_S X) \to X$.
The identity morphism $i : X \to \sF$
is simply the closed immersion into the formal completion.
We denote $\sF$ by $\sF_{dR}$ and the associated stack over $S$, by
$X_{dR} \to S$, in this case.
\end{exm}

\begin{exm}[Dolbeault Stack]
\label{exm:Dol-st}
If $X \to S$ is again separated, then the diagonal
$\Delta_{X/S} : X \to X \times_S X$ is again a closed immersion. Furthermore,
any section of a separated morphism is a closed immersion. Since the
projection $T(X \times_S X) \to X \times_S X$ of the tangent bundle of
$X \times_S X$ is affine and hence separated, the zero section
$0_X : X \to T(X \times_S X)$ is a closed immersion. Therefore, the composite
\[
\Delta' : X \to[\Delta_{X/S}] X \times_S X \to[0_X] T(X \times_S X)
\]
is a closed immersion.
We can then take $\sF \to T(X \times_S X)$ to be the
formal completion of $T(X \times_S X)$ along the set theoretic image
$\Delta'(X)$.
Then, the map $\sF \to[{(s, t)}] X \times_S X$ is obtained by
composing with the bundle projection.
The composition morphism $c : \sF \times_X \sF \to \sF$ is induced by
the addition morphism
\[
+ : T(X \times_S X) \times_{X \times_S X} T(X \times_S X) \to T(X \times_S X)
\]
The identity morphism $i : X \to \sF$
is again the closed immersion into the formal completion.
We denote $\sF$ by $\sF_{Dol}$ and the associated stack over $S$ by
$X_{Dol} \to S$, in this case.
\end{exm}

\begin{exm}[Hodge Stack]
\label{exm:Hodge-st}
Take a separated morphism of schemes $X \to T$ (for us, $T$ will
mainly be $\Spec(k)$),
consider a formal groupoid over $S := \bA^1_T = \bA^1_\bZ \times_\bZ T$
whose scheme of objects
is $pr_2 : X \times_T \bA^1_T \to \bA^1_T$. The formal scheme of morphisms is
defined as follows. We take the blow up $B \to X \times_T X \times_T \bA^1_T$ of
$(X \times_T \bA^1_T) \times_{\bA^1_T} (X \times_T \bA^1_T)
\cong (X \times_T X \times_T \bA^1_T)$ along the set theoretic image of the
map
$\Delta_{X/T} \times_T 0 : X \cong X \times_T T
\to X \times_T X \times_T \bA^1_T$ where
$0 : T \to \bA^1_T$ is the pullback of the map
$\Spec(\bZ[x] \to \bZ : x \mapsto 0)$ along the structure map
$T \to \Spec(\bZ)$,
and the blow up $B' \to X \times_T X$ of $X \times_T X$ along the image of
$\Delta_{X/T}$.
These fit into a commutative square:
\[\begin{tikzcd}
B' \ar[r] \ar[d] &
B \ar[d] \\
X \times_T X  \ar[r, "{(\id, 0)}"] & X \times_T X \times \bA^1_T
\end{tikzcd}\]
where the top and bottom arrows are both closed embeddings.
Then, $B'$ is the strict transform of $\im(\id_{X \times_T X}, 0)$
in the blow up of $X \times_T X \times_s \bA^1_T$ along
$\im(\Delta_{X/T} \times_T 0)$ by
\cite[\href{https://stacks.math.columbia.edu/tag/080E}{Lemma 080E}]
{stacks-project}.
We take $Y$ to be the complement of $B'$ in $B$ --- note that
$Y$ is an open subscheme of $B$ since the image of $B'$ is closed.
We choose a closed embedding $\Delta' : X \times_T \bA^1_T \to Y$ making
the following diagram commute:
\[\begin{tikzcd}[column sep=huge]
& Y \ar[d] \\
X \cong X \times_T T
  \ar[r, "\Delta_{X/T} \times_T 0" below]
  \ar[ru, "\Delta'" above left] &
X \times_T X \times_T \bA^1_T
\end{tikzcd}\]
Then, the formal scheme of morphisms is taken to be the formal completion
$\sF \to Y$ of $Y$ along $\Delta'(X)$. The map
$\sF \to[{(s, t)}] X \times_T X \times_T \bA^1_T$ is given by the two
projections $Y \to X \times_T X \times_T \bA^1_T \to X \times_T \bA^1_T$.
The composition map $c : \sF \times_{X \times_T \bA^1_T} \sF \to \sF$
is given by composition with the map:
\[
(X \times_T X \times_T \bA^1_T) \times_{X \times_T \bA^1_T}
(X \times_T X \times_T \bA^1_T) \to (X \times_T X \times_T \bA^1_T)
\]
The identity morphism $i : X \times_T \bA^1_T \to \sF$
is again the closed immersion into the formal completion.
We will write $\sF$ as $\sF_{Hod}$ and the associated stack over $S = \bA^1_T$
as $X_{Hod} \to \bA^1_T$, in this case.
\end{exm}

These stacks are related to each other by the following well known fact which
is the main connection with non-Abelian Hodge theory.

\begin{prop}[{\cite[33]{NonAbHodgeFilt}}]
Let $X \to T$ be a separated morphism of schemes and consider
$X_{dR}$ and $X_{Dol}$ by taking $S = T$ in \cref{exm:dR-st} and
\cref{exm:Dol-st}, respectively.
Also consider $X_{Hod}$ by taking $S = \bA^1_T$ in
\cref{exm:Hodge-st}. Then, for any closed point $\lambda : T \to \bA^1_T$,
the fibre $X_{Hod, \lambda}$ of $X_{Hod} \to \bA^{1}_T$ over $\lambda$ is
equivalent as a stack to $X_{dR}$ when $\lambda \neq 0$, and is equivalent to
$X_{Dol}$, when $\lambda = 0$.
\end{prop}

We also record here some basic facts about these stacks that make it difficult
to apply some of the techniques of \cite{ModQuivBun} directly.

\begin{prop}\label{prop:Hodge-st-diag-non-rep}
Taking $T = \Spec(k)$ for some algebraically closed field $k$ in the
context of \cref{exm:Hodge-st},
$X_{Hod}$ does not have a diagonal representable
by algebraic spaces, and is, hence, not a formal algebraic stack.
\end{prop}
\begin{proof}
We consider a $T= \Spec(k)$--point
$x \in X_{Hod}(\Spec(k))$ that factors as a map
$\Spec(k) \to[{(x', 0)}] X \times_k \bA^1_k \to X_{Hod}$,
and its stabilizer $\stb(x)$. Recalling that the preimage of a point in a
blow-up is a projective space, we have that the stabilizer
fits into the following pasting of Cartesian squares:
\[\begin{tikzcd}
\stb(x) \ar[r] \ar[d] \ar[rd, phantom, "\lrcorner" very near start] &
\bP^n_k \setminus \bP^{n - 1}_k \ar[r] \ar[d]
  \ar[rd, phantom, "\lrcorner" very near start]&
\Spec(k) \ar[d, "{(x', x', 0)}"] \\
\sF_{Hod} \ar[r] & Y \ar[r] & X \times_k X \times_k \bA^1_k
\end{tikzcd}\]
Now, $\sF_{Hod}$ is the formal completion of $Y$ along $\Delta'(X)$, while
the preimage of $\Delta'(X)$ in $\bP^n_k \setminus \bP^{n - 1}_k$
is a single point $x''$ lying over $(x', x', 0)$. By the compatibility of
formal completions with fibre products
\cite[\href{https://stacks.math.columbia.edu/tag/0APV}{Lemma 0APV}]
{stacks-project}, we have that $\stb(x)$ is the formal completion
of $\bP^{n}_k \setminus \bP^{n - 1}_k$ along a point, which is just
the formal completion of an affine chart containing that point along that
point. That is, up to change of coordinates, we have:
\[
\stb(x) = \Spf(k[[x_0, \dots, x_n]])
\]
where $\Spf$ denotes the formal spectrum functor. If $\stb(x)$ were
representable by an algebraic space, then it would have to be a scheme
by \cite[Corollary 3.1.2]{NagataComp} as the reduction $\stb(x)_{red}$ is
$\Spec(k)$. $\stb(x)$ would further have to be an affine scheme by
\cite[\href{https://stacks.math.columbia.edu/tag/06AD}{Lemma 06AD}]
{stacks-project}.
However, it is easy to see that
$\Spf(k[[x_0, \dots, x_n]])$ is not an affine scheme: as a locally
ringed space, it consists of a single point whose stalk is
$k[[x_0, \dots, x_n]]$. For any ring $A$ where $|\Spec(A)|$ is a point,
the stalk at that point must be $A$. If $\stb(x) = \Spec(A)$,
then we must have $A = k[[x_0, \dots, x_n]]$, but the latter is a DVR
and hence has two prime ideals: namely $(0)$ and $(x_0, \dots, x_n)$. This is a
contradiction. Thus, $\stb(x)$ cannot be an algebraic space
and $X_{Hod}$ cannot have diagonal representable by algebraic spaces.
\end{proof}

\begin{prop}\label{prop:Dol-st-diag-non-rep}
In the context of \cref{exm:Dol-st}, taking $S = \Spec(k)$ for some
algebraically closed field $k$ and $X \to \Spec(k)$ to be smooth,
the stack $X_{Dol}$ does not have a diagonal representable by algebraic spaces,
and is, hence, not a formal algebraic stack.
\end{prop}
\begin{proof}
Consider a point $x \in X_{Dol}(\Spec(k))$. This gives a point
$\iota \circ x : \Spec(k) \to X_{Dol} \hto X_{Hod}$, where the second map
is the closed immersion including $X_{Dol}$ as the fibre over $0 \in \bA^1_k$
in $X_{Hod}$.
We then observe that for any pair of morphisms $f, g : Y \to \Spec(k)$,
$\iota \circ x \circ f = \iota \circ x \circ g$ implies $x \circ f = x \circ g$,
since $\iota$ is a monomorphism
\cite[\href{https://stacks.math.columbia.edu/tag/0504}{Lemma 0504}]
{stacks-project}. Thus, we must have a unique map
$h : Y \to \stb(x)$ such that the two composites
$Y \to[h] \stb(x) \to \Spec(k)$ are equal. That $\iota$ is a monomorphism
also gurantees $h$ is the unique map satisfying this. This shows that
$\stb(x)$ satisfies the same universal property as $\stb(\iota \circ x)$.
Hence, by the computation of the stabilizers of $X_{Hod}$ in
\cref{prop:Hodge-st-diag-non-rep},
we have that $\stb(x) \cong \Spf(k[[x_0, \dots, x_n]])$. Thus,
$X_{Dol}$ cannot have a diagonal representable by algebraic spaces.
\end{proof}

\begin{rmk}\label{rmk:dR-st-set-fibres}
The stack $X_{dR}$ of \cref{exm:dR-st} is a stack whose fibre categories are
setoids (a groupoid that is also a preorder, and hence a groupoid with
contractible connected components). That is, it is equivalent to the
Grothendieck construction of a sheaf of sets on $\Aff_{/S}$.
This follows from the fact that the $\sF_{dR} \to X \times_S X$ is
a monomorphism of sheaves of sets. Thus, $X_{dR}$ has trivial stabilizers but
it is still not algebraic, in general.
\end{rmk}

\begin{rmk}
The lack of algebraicity of $X_{dR}, X_{Dol}, X_{Hod}$ prevent us from
concluding that $\sV(X_{dR}), \sV(X_{Dol}), \sV(X_{Hod})$ are algebraic.
Together, these prevents us from using \cite[Theorem 1.2]{HR19} (or, even more
general results such as \cite[Theorem 5.1.1]{HP23}) to conclude that
the mapping stacks involved in the definition of
$\sM_1(X_{dR}), \sM_1(X_{Dol}), \sM_1(X_{Hod})$ are algebraic.
Nevertheless, we will see that there is an alternate
route to proving their algebraicity.
\end{rmk}

\subsection{Moduli Stacks of Vector Bundles and Arrow Bundles}

We now proceed to show the main results on this paper. For this, we first
recall that ($\lambda$--)connections, connections and Higgs bundles can
all be formulated as modules over certain sheaves of algebras of differential
operators \cite[\S 2]{ModRepFunGrpI}. This formulation allows us to connect
vector bundles on formal groupoids with connections via
Simpson's crystallization functors which we now discuss.

\begin{prop}[{\cite[Theorem 5.1]{GeomNonAbHodgeFilt}}]\label{prop:diff-op-ring}
Let $(X, \sF)$ be a smooth formal category. Then, there exists
a filtered $\mcO_X$--algebra $\Lambda_\sF$ such that there is an equivalence
of categories:
\[
\Diamond_{X, \sF} : \LMod_{flf}(\Lambda_\sF) \to[\simeq] \Vect(X_\sF)
\]
where the left side is the category of left $\Lambda_\sF$--modules that is
finite locally free with respect to the induced $\mcO_X$--module structure.
\end{prop}
\begin{proof}
Strictly speaking, Simpson's result gives an $(\infty, 1)$--equivalence:
\[
\Perf_{\mcO_X}(\Lambda_\sF) \to[\simeq] \Perf(X_\sF)
\]
where the left side is the $(\infty, 1)$--category of complexes of
$\Lambda_\sF$--modules quasi-isomorphic as complexes of $\mcO_X$--modules
to complexes whose terms are finite locally free $\mcO_X$--modules.
To get the equivalence stated above, we first observe that the functor
realizing this equivalence is first defined for the case of affine $X$
and $\Lambda_X$--modules on one side, and $\mcO_{X_\sF}$--modules on the other,
and then observing that applying the functor term-wise gives an extension
to complexes \cite[\S 3.5, p. 24]{GeomNonAbHodgeFilt}. The global case
is obtained by taking the limit of the crystallization functors over a Zariski
cover of $X$ \cite[Proof of Theorem 5.1]{GeomNonAbHodgeFilt}.
This implies that the functor preserves complexes concentrated in degree zero.
\end{proof}

\begin{defn}\label{defn:diff-op-ring}
We call $\Lambda_\sF$ in \cref{prop:diff-op-ring} the sheaf of rings
of differential operators associated to $(X, \sF)$. We call $\Diamond_{X, \sF}$
the crystallization functor.
We will suppress the subscript of the crystallization functor, when there is
no confusion.
\end{defn}

\begin{rmk}
$\Lambda_\sF$ is a sheaf of rings of differential operators as defined in
\cite[\S 2]{ModRepFunGrpI}.
\end{rmk}

\begin{prop}\label{prop:cryst-pullback-compat}
In the context of \cref{defn:diff-op-ring}, consider two formal groupoids
$(X, \sF), (Y, \sG)$ with an internal functor
$(f_0, f_1) : (X, \sF) \to (Y, \sG)$ giving a morphism of quotient stacks
$f : X_\sF \to Y_\sG$.
Then, the following diagram of categories commutes up to natural isomorphism:
\[\begin{tikzcd}
\LMod_{flf}(\Lambda_\sG) \ar[r, "\Diamond_{Y, \sG}"] \ar[d, "f_0^*" left] &
\Vect(Y_\sG) \ar[d, "f^*"] \\
\LMod_{flf}(\Lambda_\sF) \ar[r, "\Diamond_{X, \sF}" below] &
\Vect(X_\sF)
\end{tikzcd}\]
\end{prop}
\begin{proof}
We pick a Zariski cover $\set{X_i \to X}_{i \in I}$, and then take
the formal groupoids $(X_i, \sF_i)$ where
$\sF_i = X_i \times X_i \times_{X \times X} \sF$. We define
the formal groupoids $(Y_i, \sG_i)$ analogously.
Without loss of generality, we can take the $X_i$ to be $X \times_Y Y_i$.
We get corresponding morphisms of stacks
$X_{i, \sF_i} \to X_\sF, X_{i, \sF_i} \to[f_i] Y_{i, \sF_i},
Y_{i, \sG_i} \to Y_{\sG}$.
Then, for an object $(E, \phi) \in \LMod_{flf}(\Lambda_\sG)$, we have:
\begin{align*}
& \Diamond_{X, \sF}(f_0^*(E, \phi)) \\
=& \lim_{X_i \to X} \Diamond_{X_i, \sF_i}((f_0^*(E, \phi))|_{X_i}) \\
=& \lim_{X_i \to X} \Diamond_{X_i, \sF_i}(f_0^*((E, \phi)|_{Y_i}))
\end{align*}
Let $A_i, B_i$ be the rings of global sections of $X_i, Y_i$ so that
$X_i = \Spec(A_i), Y_i = \Spec(B_i)$.
Let $N_i$ denote the $B_i$--module $\Gamma(E|_{Y_i})$, and
$M_i = A_i \otimes_{B_i} N_i$, the $A_i$--module $\Gamma(f^*_0(E|_{Y_i}))$.
If we define
$F_i := \Gamma(\Lambda_{\sF_i})^\vee$
There is an equivalence
\cite[\S 3.5]{GeomNonAbHodgeFilt} between
$\Gamma(\Lambda_{\sF_i})$--modules and connections as defined
in \cite[\S II.1.2]{CohomCrysCharP} and
by the definition of base change for connections given in
\cite[94, \S II.1.2.5]{CohomCrysCharP}, we can deduce:
\[
\Diamond_{X_i, \sF_i}(f^*_0((E, \phi)|_{Y_i}))
\cong f^*_0\br{\Diamond_{X_i, \sF_i}\br{\br{E, \phi}|_{Y_i}}}
\]
Finally, since limits of functors are defined objectwise and $f^*_0$
is given by precomposition, $f^*_0$ commutes with limits. Thus, we have:
\[
\Diamond_{X, \sF}(f_0^*(E, \phi))
\cong f_0^*\br{\lim_{X_i \to X} \Diamond_{X_i, \sF_i}((E, \phi)|_{Y_i})}
\]
A similar argument applies for morphisms.
\end{proof}

These results allow us to give a concrete description of the moduli prestack
of arrow bundles on a smooth formal groupoid. At the same time, we get
a simple description of the fibres of the boundary map as in
\cref{prop:bndry-map-fibre-simple}.

\begin{thm}\label{thm:mod-arr-bun-formal-grpd-concrete}
Given a smooth formal groupoid $(X, \sF)$, where the morphism $X \to S$
is quasi-separated, the underlying category of $\sM_1(X_\sF)$ is equivalent
to a category defined by the following data:
\begin{itemize}
\item Objects are tuples $(u, E, \phi, F, \psi, s)$ where:
  \begin{itemize}
  \item $u : U \to S$ is an object of $\Aff_{/S}$
  \item $(E, \phi), (F, \psi)$ are
  $\Lambda_{\sF, U} := pr_1^*\Lambda_{_\sF}$--modules, where
  $E, F$ are finite locally free with respect to the induced
  $\mcO_{U \times \mcX}$--module structure
  \item $s : E \to F$ is a morphism of $\Lambda_{U \times \mcX}$--modules
  \end{itemize}
\item Morphims $(u, E, \phi, F, \psi, s) \to (u', E', \phi', F', \psi', s')$
  are tuples $(f, a, b)$ where:
  \begin{itemize}
  \item $f : U = \dom(u) \to U' = \dom(u')$ is a morphism in $\Aff_{/S}$
  \item $a : f^<E' \to E, b^<F' \to F$ are morphisms of
    $\Lambda_{\sF, U}$--modules such that the following diagram commutes:
    \[\begin{tikzcd}
    f^<E' \ar[r, "a"] \ar[d, "f^<s'" left] & E \ar[d, "s"] \\
    f^<F' \ar[r, "b" below] & F'
    \end{tikzcd}\]
    noting that $f^<\Lambda_{\sF, U'} = f^<(u')^<\Lambda_{X_\sF} =
    pr_1^*\Lambda_{X_\sF} = \Lambda_{\sF, U}$.
  \end{itemize}
\end{itemize}
\end{thm}
\begin{proof}
Let $(u, E_0, F_0, s_0)$ be an object of $\sM_1(X_\sF)$ as defined in
\cref{cns:mod-prest-arr-bun}. Then, the inverse of the crystallization functor
for $(U \times X, U \times \sF)$ applied to $s_0$ gives a morphism of
$\Lambda_{U \times X_\sF}$--modules
\[
s = \Diamond_{U \times X, U \times \sF}^{-1}(s_0) :
(E, \phi) = \Diamond_{U \times X, U \times \sF}^{-1}(E_0)
\to (F, \psi) = \Diamond_{U \times X, U \times \sF}^{-1}(F_0)
\]
where $\phi, \psi$ are the structure maps making the
$\mcO_{U \times \mcX}$--modules $E, F$ into
$\Lambda_{U \times X_\sF}$--modules. Thus, $(u, E, \phi, F, \psi, s)$
is a tuple giving an object as described in the statement of the theorem.

Now, let $(f, a_0, b_0) : (u, E_0, F_0, s_0) \to (u', E_0', F_0', s_0')$
be a morphism in $\sM_1(X_\sF)$. Then, $a_0 : f^<E_0' \to E_0,
b_0 : f^<F_0' \to F_0$ are morphisms of $\mcO_{X_\sF}$--modules. Then,
we get isomorphisms of $\Lambda_{\sF}$--modules:
\begin{align*}
a' := \Diamond^{-1}(a_0) :
E'' = \Diamond^{-1}(f^<E_0') \to E = \Diamond^{-1}(E_0) \\
b' := \Diamond^{-1}(b_0) :
F'' = \Diamond^{-1}(f^<F_0')\to F = \Diamond^{-1}(F_0)
\end{align*}
We also obtain isomorphisms
$\psi_E : f^<\Diamond^{-1}(E_0') \to[\cong] \Diamond^{-1}(f^<E_0')$,
$\psi_F : f^<\Diamond^{-1}(F_0') \to[\cong] \Diamond^{-1}(f^<F_0')$ by
\cref{prop:cryst-pullback-compat}, giving us isomorphisms
$a := a_0 \circ \psi_E, b := b_0 \circ \psi_F$. We denote
$E' := \Diamond^{-1}(E_0'), F' := \Diamond^{-1}(F_0'), s' := \Diamond^{-1}(s_0')$.
By naturality of the $\psi_E, \psi_F$, we have the following commuting diagram:
\[\begin{tikzcd}
f^<E' \ar[r, "\psi_E"] \ar[d, "f^<s'" left] \ar[rr, bend left, "a" above] &
\Diamond^{-1}(f^<E_0') \ar[r, "a'"] \ar[d,"\Diamond^{-1}(f^<s'_0)" left] &
E \ar[d, "s"] \\
f^<F' \ar[r, "\psi_F" below] \ar[rr, bend right, "b" below] &
\Diamond^{-1}(f^<F_0') \ar[r, "b'"] &
F
\end{tikzcd}\]
This shows that $(f, a, b)$ is a morphism of the form described in the statement.

If we now set $\Psi(u, E_0, F_0, s)$ and $\Psi(f, a_0, b_0) := (f, a, b)$,
it is not hard to check that $\Psi$ is a functor and it commutes with the
projections to $\Aff_{/S}$ by construction. It is also not hard to check that
it is an equivalence by using the fact that on fibre categories, it is simply
the crystallization functor which is an equivalence by \cref{prop:diff-op-ring}.
\end{proof}

\begin{prop}\label{prop:frml-grpd-bndry-map-fibre}
In the context of \cref{defn:arr-bun-fixed-src-trgt}, setting $\mcY = X_\sF$ for
a smooth formal groupoid $(X, \sF)$,
$\Diamond^{-1}(E) = E_\diamond, \Diamond^{-1}(F) = F_\diamond$ and
$\Lambda_{\sF, U} := pr_1^*\Lambda_{\sF}$, we have that $P^{X_\sF}_{v, w}$
is equivalent to the category with the following description:
\begin{itemize}
\item Objects are tuples $(r, s)$ where:
\begin{itemize}
\item $r : V \to U$ is a morphism in $\Aff_{/S}$
\item $s_\diamond : r^<E_\diamond \to r^<F_\diamond$ is a morphism of
$r^<\Lambda_{\sF, U}$--modules
\end{itemize}
\item Morphisms $(r, s_\diamond) \to (r', s_\diamond')$ are morphisms
$f : V = \dom(r) \to V' = \dom(r')$ in $\Aff_{/S/U}$ such that
$f^<s' = s$.
\end{itemize}
\end{prop}
\begin{proof}
We take $\mcY = X_\sF$ in \cref{prop:bndry-map-fibre-simple}, and consider
an object $(r, s)$ in $Q^{X_\sF}_{v, w} \simeq P^{X_\sF}_{v, w}$ according to
the description there. Then,
$\Diamond^{-1}(s) : \Diamond^{-1}(r^<E) \to \Diamond^{-1}(r^<F)$ is a
morphism of $r^<\Lambda_{\sF, U}$--modules. Choosing a natural
isomorphism $\psi$ implementing the $2$--commutativity in
\cref{prop:cryst-pullback-compat}, we get a unique morphism of
$r^<\Lambda_{\sF, U}$--modules $s_\diamond$ making the following square commute:
\[\begin{tikzcd}
\Diamond^{-1}(r^<E) \ar[r, "\psi_{E}"] \ar[d, "\Diamond^{-1}(s)" left] &
r^<E_\diamond \ar[d, "s_\diamond"] \\
\Diamond^{-1}(r^<F) \ar[r, "\psi_{F}" below] &
r^<F_\diamond
\end{tikzcd}\]
Then, $(r, s_\diamond)$ is an object as described in the statement.

Now, consider a morphism $f : (r, s) \to (r', s')$ in $P^{X_\sF}_{v, w}$.
Then, $f^<s' = s$ implies $\Diamond^{-1}(f^<s') = \Diamond^{-1}(s)$.
By the uniqueness of $s_\diamond$ and $s_\diamond'$ as in the previous paragraph,
we must have $f^<s_\diamond' = s_\diamond$.

From this description, it is easy to see that the assignment
$(r, s) \mapsto (r, s_\diamond), f \mapsto f$ is a fully faithful functor
from the $P^{X_\sF}_{v, w}$ to the category described in the statement.
That this functor is also essentially surjective follows from the
fully faithfulness of the crystallization functors --- note that the fibre
categories are simply the Hom sets $\Hom_{\mcO_{U \times X_\sF}}(r^<E, r^<F)$
on one side and $\Hom_{r^<\Lambda_{\sF, U}}(r^<E_\diamond, r^<F_\diamond)$,
on the other.
\end{proof}

With this, we can prove the algebraicity of the moduli stack of
arrow bundles on a smooth formal groupoid whenever the moduli stacks of vector
bundles, $\Lambda_\sF$--modules and arrow bundles on the object scheme of the
formal groupoid are algebraic.

\begin{thm}\label{thm:mod-st-arr-bun-formal-grpd-alg}
Let $(X, \sF)$ be a smooth formal groupoid over $S$, where the map $X \to S$ is
smooth, and such that $\sM(X), \sM_1(X)$ and $\sM(X_\sF)$ are all algebraic.
Then, the following are true:
\begin{enumerate}
\item The boundary map $\partial_{X_\sF} : \sM_1(X_\sF) \to \sM(X_\sF)^2$ is
representable by algebraic spaces. Therefore, $\sM_1(X_\sF)$ is an algebraic
stack.
\end{enumerate}
If, in addition, $X \to S$ is projective, then
\begin{enumerate}[resume]
\item $\partial_{X_\sF}$ is affine and of finite presentation.
\end{enumerate}
If, furthermore, $\sM(X_\sF)$ is (locally) of finite presentation, then
\begin{enumerate}[resume]
\item $\sM_1(X_\sF)$ is (locally) of finite presentation.
\end{enumerate}
\end{thm}
\begin{proof}
We will show that the fibre $P^{X_\sF}_{v, w}$ of any map
$(v, w) : U \to \sB_1(X_\sF)^2$ with $U \in \Aff_{/S}$ is equivalent to an
equalizer of algebraic spaces and is, hence, an algeraic space. For this, we
first recall from \cite[\S 2]{ModRepFunGrpI} that for any
$z : Z \to S$, $\Lambda_{\sF, Z} := pr_1^*\Lambda_{\sF}$ is a filtered
ring sheaf and from \cite[Lemma 2.2]{ModRepFunGrpI}, that the stages of the
filtration are coherent. Then, we also know that $\Lambda_{\sF, Z}$ is
split almost polynomial \cite[\S 3.3, \S 5.1]{GeomNonAbHodgeFilt} so that
$\Lambda_{\sF, Z, 1}$ is locally free by
\cite[Theorem 2.11]{ModRepFunGrpI}.
This means $\Lambda_{\sF, Z, 1}$ is finite locally free.
At the same time, the split almost polynomial condition implies that
there is an open cover $\coprod_{j \in J} Z_i \to Z \times X$ such that
$\Lambda_{\sF, Z}|_{Z_j}$ is generated by $\Lambda_{\sF, U, 1}|_{Z_j}$
\cite[Proof of Theorem 2.11]{ModRepFunGrpI}. Since morphisms of sheaves
satisfy descent, we conclude that a morphism $s : (E, \phi) \to (F, \psi)$
is a morphism of $\Lambda_{\sF, Z}$--modules if and only if the following
diagram commutes:
\[\begin{tikzcd}
\Lambda_{\sF, Z, 1} \otimes E
\ar[r, "\phi"] \ar[d, "\id_{\Lambda_{\sF, Z, 1}} \otimes s" left] &
E \ar[d, "s"] \\
\Lambda_{\sF, Z, 1} \otimes F \ar[r, "\psi" below] & F
\end{tikzcd}\]

Now, consider the $\Lambda_{\sF, U}$--modules $(E_0, \phi), (F_0, \psi)$
corresponding to $v, w$ respectively. Then, we define two functors
$\rho_1, \rho_2 : P^X_{E, F} \to P^X_{\Lambda_{\sF, U, 1} \otimes E, F}$
as follows. $\rho_1$ sends an object $(r, s)$ to the composite
\[
r^<\Lambda_{\sF, Z, 1} \otimes r^<E \to[r<\phi] r^<E \to[s] r^<F
\]
Then, given a morphism $f : (r, s) \to (r', s')$, we have $f^<s' = s$, so that
$f^<(s' \circ (r')^<\phi) = f^<s' \circ f^<(r')^<\phi = s \circ r^<\phi$. This
shows that $f$ is also a morphism in $P^X_{\Lambda_{\sF, U, 1} \otimes E, F}$.
We set $\rho_1(f) = f$.
It is straightforward to verify that these mappings make $\rho_1$ a functor
commuting with the projection to $\Aff_{/S}$. Thus, it is a morphism of
prestacks.
The map $\rho_2$ sends $(r, s)$ to the composite
\[
r^<\Lambda_{\sF, Z, 1} \otimes r^<E \to[\id \otimes s]
r^<\Lambda_{\sF, Z, 1} \otimes r^<F \to[r^<\psi]
r^<F
\]
If $f : (r, s) \to (r', s')$ is a morphism, then $f^<s' = s$ again, and this
implies:
\[
f^<((r')^<\psi \circ (\id \otimes s'))
= f^<(r')^<\psi \circ f^<(\id \otimes s')
= f^<(r')^<\psi \circ (f^<\id \otimes f^<s')
= r^<\psi \circ (\id \otimes s)
\]
so that $f$ is again a morphism in $P^X_{\Lambda_{\sF, U, 1} \otimes E, F}$.
We again set $\rho_2(f) = f$, and can check that this gives a morphism of
prestacks.

We can then see that $P^{X_\sF}_{v, w}$ fits into the following
equalizer diagram:
\[\begin{tikzcd}
P^{X_\sF}_{v, w} \ar[r] &
P^X_{E, F}
  \ar[r, shift left, "\rho_1" above]
  \ar[r, shift right, "\rho_2" below] &
P^{X}_{\Lambda_{\sF, U, 1} \otimes E, F}
\end{tikzcd}\]
where the first map is the forgetful functor sending objects and morphisms
to themselves. To see this, we notice that the equalizer is strict since
all prestacks involved are fibred in $\Set$. Then, by
\cref{prop:frml-grpd-bndry-map-fibre}, the objects
and morphisms of $P^{X_\sF}_{v, w}$ are a subset of those of $P^X_{E, F}$
satisfying the commutativity constraint for being module homomorphisms. The
equality of the maps $\rho_1, \rho_2$ is precisely this condition.
It remains to show that
$P^X_{E, F}, P^X_{\Lambda_{\sF, U, 1} \otimes E, F}$ are algebraic spaces.
This follows from \cref{prop:bndry-map-repr-alg-sp}. This shows point (i).

In the case that
$X \to S$ is projective, \cref{thm:X-fppf-proj-bndry-map-aff-fp} applies
and shows that $P^X_{E, F}, P^X_{\Lambda_{\sF, U, 1} \otimes E, F}$ are
affine and of finite presentation over $U$. Since an equalizer over $S$
is also an equalizer over $U$, and the properties of being affine and of
finite presentation are preserved under finite limits, we must have that
$P^{X_\sF}_{v, w}$ is affine and of finite presentation
over $S$. This shows point (ii).

For (iii), we observe that $\partial_{X_\sF}$ being of finite
presentation and $\sM(X_\sF)$ being (locally) of finite presentation implies
that the morphism $\sM_1(X_\sF) \to[\partial_{X_\sF}] \sM(X_\sF)^2 \to S$
is (locally) of finite presentation.
\end{proof}

We will now recall that $\sM(X_\sF)$ is well-behaved under some mild conditions,
which is well-known \cite[Theorem 6.13]{GeomNonAbHodgeFilt},
and show that some of the good behaviour of $\sM(X_\sF)$ transfers
over to $\sM_1(X_\sF)$.

\begin{prop}\label{prop:mod-Lambda-forget-aff-loc-fp}
For a smooth formal groupoid $(X, \sF)$ over $S$, if the morphism
$X \to S$ is smooth and projective, then
the forgetful functor $\sM(X_\sF) \to \sM(X)$ is affine and locally of
  finite presentation.
In particular, $\sM(X_\sF)$ is algebraic, locally of finite presentation and has
  affine diagnoal over $S$.
\end{prop}
\begin{proof}
We first show that the forgetful map $\sM(X_\sF) \to \sM(X)$ is affine and of
finite presentation. To see this, we observe that for any morphism
$v : U \to \sM(X)$ for some object $\Aff_{/S}$ corresponding to a finite locally
free $\mcO_{U \times X}$--module $E$,
we consider the fibre product $\sM(X_\sF; E) := \sM(X_\sF) \times_{\sM(X)} U$.
Unwrapping the definition of fibre products, the underlying category
of this stack has the following description:
\begin{itemize}
\item Objects are tuples $(v, F, \phi, \alpha)$ where:
  \begin{itemize}
  \item $v : V \to U$ is an object in $\Aff_{/S/U}$.
  \item $F$ is a finite locally free sheaf on $V \times X$.
  \item $\phi : \Lambda_{\sF, V} \otimes F \to F$ is a $\Lambda_{\sF, V}$--module
    structure on $F$.
  \item $\alpha : v^<E \to F$ is an isomorphism of locally free sheaves.
  \end{itemize}
\item Morphisms $(v, F, \phi, \alpha) \to (v', F', \phi', \alpha')$ are tuples
  $(f, a)$ where:
  \begin{itemize}
  \item $f : V = \dom(v) \to V' = \dom(v')$ is a morphism over $U$.
  \item $a : f^<F' \to F$ is an isomorphism of $\Lambda_{\sF, V}$--modules, such
    that the following diagram commutes:
  \[\begin{tikzcd}
  v^<E = f^<(v')^<E \ar[r, "f^<\alpha'"] \ar[d, "\id" left] &
  f^<F' \ar[d, "a"] \\
  v^<E \ar[r, "\alpha" below] &
  F \\
  \end{tikzcd}\]
  \end{itemize}
\end{itemize}
We observe that each object $(v, F, \phi, \alpha)$ of this category is
isomorphic to the object
\[
(v, v^<E, \alpha^{-1} \circ \phi \circ (\id_{\Lambda_{\sF, V}} \otimes \alpha),
  \id_{v^<E})
\]
via the morphism $(\id_V, \alpha)$. Next, consider a morphsim
\[
(f, a) : (v, v^<E, \phi, \id_{v^<E}) \to (v', (v')^<E, \phi', \id_{(v')^<E})
\]
The commutativity condition in the definition of morphisms becomes
\[
a \circ f^<\id_{(v')^<E} = \id_{v^<E} \circ \id_{v^<E}
\iff a = \id_{v^<E}
\]
which implies that $f^<\phi' = \phi$, by the fact that $a$ is a morphism
$\Lambda_{\sF, V}$--modules. This shows that the underlying category is, in
turn, equivalent to the category with the following description:
\begin{itemize}
\item Objects are tuples $(v, \phi)$ where:
  \begin{itemize}
  \item $v : V \to U$ is an object in $\Aff_{/S/U}$.
  \item $\phi : \Lambda_{\sF, V} \otimes v^<E \to v^<E$ is a
    $\Lambda_{\sF, V}$--module structure on $F$.
  \end{itemize}
\item Morphisms $(v, \phi) \to (v', \phi')$ is a morphism
  $f : V = \dom(v) \to V' = \dom(v')$ over $U$ such that $f^<\phi' = \phi$.
\end{itemize}

A very similar argument as in the proof of
\cref{thm:mod-st-arr-bun-formal-grpd-alg} shows that this is a finite limit
involving $P^X_{\Lambda_{\sF, U, 1}^{\otimes 2} \otimes E, E}$,
$P^X_{\Lambda_{\sF, U, 1} \otimes E, E}$,
$P^X_{\mcO_{U \times X} \otimes E, E}$, that captures the commutativity
constraints defining the structure map of a module sheaf, and hence is an
affine scheme of finite presentation by \cref{thm:X-fppf-proj-bndry-map-aff-fp}
and the preservation of these properties under finite limits.
Since the morphism $\sM(X_\sF) \to \sM(X)$ is affine and of finite presentation,
while $\sM(X)$ is algebraic locally of finite presentation and has affine
diagonal by \cite[Theorem 1.0.1]{Wang-BunG}, $\sM(X_\sF)$
is algebraic by
\cite[\href{https://stacks.math.columbia.edu/tag/05UM}{Lemma 05UM}]
{stacks-project} and has affine diagonal by \cite[Lemma 4.9]{ModQuivBun}.
Then, in the composite $\sM(X_\sF) \to \sM(X) \to S$, the first morphism is of
finite presentation and the second one is locally of finite presentation by
\cite[Theorem 1.0.1]{Wang-BunG} again, showing that $\sM(X_\sF)$ is locally
of finite presentation over $S$.
\end{proof}

\begin{rmk}
This is likely a direct consequence of the fibre computation in the proof of
\cite[Theorem 6.13]{GeomNonAbHodgeFilt}, but we felt it is useful to have
a simpler proof in the realm of $1$--stacks.
\end{rmk}

\begin{thm}\label{thm:mod1-Lambda-alg-aff-diag-loc-fp}
For a smooth formal groupoid $(X, \sF)$ over $S$, if the morphism
$X \to S$ is smooth and projective, then $\sM_1(X_\sF)$ is algebraic, locally of
finite presentation and has affine diagonal.
\end{thm}
\begin{proof}
$\sM(X_\sF)$ has these properties by \cref{prop:mod-Lambda-forget-aff-loc-fp},
and hence, so does $\sM(X_\sF)^2$. Then, since the morphism $\partial_{X_\sF}$
is affine by \cref{thm:mod-st-arr-bun-formal-grpd-alg}(ii), $\sM_1(X_\sF)$
is algebraic by
\cite[\href{https://stacks.math.columbia.edu/tag/05UM}{Lemma 05UM}]
{stacks-project} and has affine diagonal by \cite[Lemma 4.9]{ModQuivBun}.
Then, the morphism $\sM(X_\sF)^2 \to \sM(X)^2$ is locally of finite presentation
by \cref{prop:mod-Lambda-forget-aff-loc-fp} again, while
$\partial_{X_\sF}$ is locally of finite presentation by
\cref{thm:mod-st-arr-bun-formal-grpd-alg}(ii). Also,
$\sM(X)^2$ is locally of finite presentation by \cite[Theorem 1.0.1]{Wang-BunG}.
This shows that the composite
\[
\sM_1(X_\sF) \to[\partial_{X_\sF}] \sM(X_\sF)^2 \to \sM(X)^2 \to S
\]
is locally of finite presentation.
\end{proof}

\begin{thm}\label{thm:mod-st-conn-alg-lfp-aff-diag}
In the case that $X \to \Spec(k)$ is a smooth and projective morphism,
the moduli stacks
$\sM_1(X_{Dol}), \sM_1(X_{dR}), \sM_1(X_{Hod})$ parametrizing
Higgs bundle morphisms, connection morphisms and $\lambda$--connection
morphisms respectively over $X$ are all agebraic, locally of finite presentation
and have affine diagonal.
\end{thm}
\begin{proof}
For the first two, take $S = \Spec(k)$, and for the third, take
$S = \bA^1_k$, and apply \cref{thm:mod1-Lambda-alg-aff-diag-loc-fp}.
\end{proof}

\begin{rmk}
The result \cite[Theorem 6.13]{GeomNonAbHodgeFilt} states that the moduli
stack of complexes of $\Lambda_\sF$--modules that are perfect as complexes
of $\mcO_X$--modules is a locally geometric $n$--stacks. However, we can unpack
the meaning of this to discover that this shows the algebraicity of the
sub--$1$--stack of complexes concentrated in degree $0$.
\end{rmk}


\subsection{Quiver Bundles on Formal Groupoids}\label{subsec:QuivBun}

In \cite[\S 4]{ModQuivBun}, for each stack $\mcY$ and each fintite simplicial
set $I$, we gave a construction of a moduli stack $\mcM_{\Vect(\mcY), I}$
parametrizing diagrams of vector bundles of shape $I$ on a stack $X$. In fact,
the construction given there is still well-defined when $\mcY$ is a prestack.
This construction was given as a finite iterated limit involving the (pre)stacks
$\sM(\mcY), \sM_1(\mcY)$ and $\mathrm{pt}$. Now, in light of
\cref{thm:arr-bun-arr-sec-equiv} and
\cref{thm:mod-arr-bun-formal-grpd-concrete}, if we take $\mcY = X_\sF$
for some formal groupoid $(X, \sF)$, by varying $\sF$ to be
$\sF_{Dol}, \sF_{dR}$ or $\sF_{Hod}$, we obtain
moduli stacks parametrizing $I$--shaped diagrams of connections,
$\lambda$--connections, or Higgs bundles.

\begin{notn}
For notational consistency with the rest of the present paper,
we will write $\mcM_{\Vect(X_\sF), I}$ as $\sM_I(X_\sF)$.
\end{notn}

Furthermore, $\sM_I(X_\sF)$ is contravariantly functorial in $I$ by
construction. Taking $I = \Delta^n$ for various $n$, we thus obtain, in
particular, a simplicial object $\sM_{\bullet}(\mcY)$ whose levels are moduli
stacks of diagrams of vector bundles of shape $\Delta^n$, and whose structure
maps are given by ``face'' and ``degeneracy'' maps. This is what we mean by the
categorification of the various sides of non-Abelian Hodge theory. In
particular, $\sM_\bullet(X_{Hod})$ is the categorification of the non-Abelian
Hodge filtration.

\begin{rmk}\label{rmk:smplcl-stacks}
However, we note that since there are no higher morphisms in the
category of sheaves over a prestack, we cannot expect these objects to hold much
important information beyond simplicial dimension $1$: that is, for sheaves over
$1$--stacks, $\sM_0(\mcY), \sM_1(\mcY)$ should contain all the important
information. The full simplicial theory becomes relevant when we are dealing
with sheaves of simplicial Abelian groups, chain complexes or spectra on
(derived) stacks, but we will defer a discussion of this point to future
work.
\end{rmk}

We now observe that $\sM_I(X_\sF)$ is well-behaved under some mild conditions
on $(X, \sF)$.

\begin{thm}\label{thm:mod-quiv-bun-formal-grpd-loc-fin-pre}
For any smooth formal groupoid $(X, \sF)$, if the morphism $X \to S$
is smooth and projective, for any finite simplicial set $I$, $\sM_I(X_\sF)$ is
algebraic, locally of finite presentation and has affine diagonal.
\end{thm}
\begin{proof}
By construction, $\sM_I(X_\sF)$ is an iterated finite limit
involving $\sM_1(X_\sF)$, $\sM(X_\sF)$ and $\mathrm{pt}$. The claim now follows
\cref{prop:mod-Lambda-forget-aff-loc-fp},
\cref{thm:mod1-Lambda-alg-aff-diag-loc-fp}
and the fact that the properties involved are preserved under taking finite
limits.
\end{proof}


\subsection{Non-Abelian Hodge Theory}\label{subsec:NonAbHodge}

We wish to have an analogue of the non-Abelian Hodge correspondence between
the simplicial stacks $\sM_\bullet(X_{dR})$ and $\sM_\bullet(X_{Dol})$.
In light of \cref{rmk:smplcl-stacks}, in this work,
we will mainly deal with the truncated simplicial objects consisting of
$\sM_1(-)$ and $\sM(-)$ and the terminal stack $\Aff_{/S}$ in higher degrees,
as we are focussing on sheaves of sets on $1$--stacks.
We first recall that the non-Abelian Hodge correspondence has some
restrictions. First, it is a bijection between moduli spaces of flat
connections on one side and Higgs bundles on the other side, which happens to
be a homeomorphism if we restrict to polystable objects with vanishing
rational Chern classes on the Higgs bundle side. It is known that
it cannot be extended
\cite[Counterexample on pages 38---39]{ModRepFunGrpII}. Furthermore,
the mapping is neither complex analytic nor smooth even though both sides have
complex structures. This suggests that we need some way to pass from algebraic
stacks to topological stacks to address non-Abelian Hodge theory.
This is possible through analytification of stacks, a simple version of which
we will now discuss.
For this purpose, let us now suppose $k = \bC$ for simplicity.

%


\begin{cns}\label{cns:analytification-mod-st-quiv-bun-formal-grpd}
Suppose $S$ is of finite presentation over $\bC$.
Given any finite simplicial set $I$ and a smooth formal groupoid $(X, \sF)$ with $X \to S$ projective and smooth, choose a smooth atlas
$a : A = \coprod_{j \in J} U_j \to \sM_I(X_\sF)$ such that the morphism $a$ is also
affine and each $U_j \in \Aff_{/S}$ is of finite presentation, using
\cref{thm:mod-quiv-bun-formal-grpd-loc-fin-pre} and
\cite[\href{https://stacks.math.columbia.edu/tag/04YF}{Lemma 04YF}]
{stacks-project}. Since $a$ is affine,
$R := A \times_{\sM_I(X_\sF)} A = \coprod_{i, j} U_i \times_{\sM_I(X_\sF)} U_j$
is a disjoint union of affine schemes $U_i \times_{\sM_I(X_\sF)} U_j$
so that $R$ is a locally separated algebraic space. Of course, $A$ is also
a locally separated algebraic space for the same reason.
Furthermore, since $\sM_I(X_\sF)$ and $A$ are locally of finite presentation,
so is $R$ since this property is preserved by fibre products of stacks.
Next, by
\cite[\href{https://stacks.math.columbia.edu/tag/04T5}{Lemma 04T5}]
{stacks-project},
$(A, R)$ gives a presentation of $\sM_I(X_\sF)$.
This allows us to consider analytifications of $A, R$ \cite{GenGAGA} to obtain a groupoid object
internal to complex analytic spaces, noting that analytification is compatible
with fibre products \cite[Theorem 2.2.3.]{NonArchAnalytification}. This, in
turn, also gives a groupoid object in topological spaces. We can take the
quotient stacks over the respective categories to get complex analytic and
topological stacks:
\[
\sM_I^{an}(X_\sF), \sM_I^{top}(X_\sF)
\]
respectively. Note that we retain the notation $X_\sF$ although we are not yet
talking about the analytification of $X$ and $\sF$.
\end{cns}

We now state our desired categorification of the non-Abelian Hodge
correspondence as a conjecture:

\begin{cnj}
Given a smooth projective variety $X$ over $k$, there is a suitable substack
\[
\sM^{top, nice}(X_{Dol}) \subset \sM_1^{top}(X_{Dol}),
\]
and mappings
\[
\sM^{top}(X_{dR}) \to \sM^{top, nice}(X_{Dol}),
\sM_1^{top}(X_{dR}) \to \sM_1^{top, nice}(X_{Dol})
\]
that induce a ``categorical'' equivalence of simplicial topological stacks
\[
\sM^{top}_{\bullet}(X_{dR}) \to[\simeq] \sM^{top, nice}_{\bullet}(X_{Dol})
\]
for some suitable meaning of ``categorical'' in context of simplicial objects
in topological stacks. This should still hold if we
replace $\sM_\bullet(X_{dR}), \sM_\bullet(X_{Dol})$ with
the respective simplicial stacks obtained by applying our constructions
in this paper with derived mapping stacks in place of ordinary mapping stacks.
\end{cnj}
\begin{proof}[Proof idea]
$\sM_\bullet^{top}(X_{dR})$
and $\sM_\bullet^{top}(X_{Dol})$ are simplicial substacks of
$\sM_\bullet(X_{Hod}) \to \bA^1_\bC$: namely, the (simplicial degree-wise)
fibres above $0, 1 \in \bA^1_\bC$ respectively. We can then try to apply the
``preferred sections'' approach as discussed in \cite[\S 4]{NonAbHodgeFilt}
for every simplicial degree. Assuming we are able to accomplish this,
we must, of course, then examine how the maps
$\sM_n^{top}(X_{dR}) \to \sM_n^{top, nice}(X_{Dol})$ interact with the
simplicial maps.
\end{proof}

 Of course, we must also make precise what we mean by
``categorical equivalence'' and then check that such a condition is satisfied.

\begin{rmk}
According to a conversation we had with Carlos Simpson, it is likely that the
above result will only be true in the derived setting.
\end{rmk}

\printbibliography

\end{document}